\documentclass[a4paper,11pt]{article}
\usepackage{float} 
\usepackage{amsmath}
\usepackage{amsthm}
\usepackage{graphicx} 
\usepackage{latexsym}
\usepackage{amsfonts}
\usepackage{amssymb}
\usepackage{xcolor}

\newtheorem{theorem}{Theorem}
\newtheorem{lemma}[theorem]{Lemma}
\newtheorem{proposition}[theorem]{Proposition}
\newtheorem{corollary}[theorem]{Corollary}

\usepackage{enumerate}

\theoremstyle{definition}

\newtheorem{remark}[theorem]{Remark}

\newcommand{\pr}{\mathbf P}
\newcommand{\E}{\mathbf E} 

\newcommand{\e}{\mathbf E}
\newcommand{\A}{\mathcal{A}}
\newcommand{\C}{\mathcal{C}}
\renewcommand{\L}{\mathcal{L}}

\newcommand{\veps}{\varepsilon}

\title{Ordered random walks and the Airy line ensemble}
\author{Denis Denisov\thanks{Department of Mathematics, University of Manchester, 
Manchester, UK.} \and Will FitzGerald\footnotemark[1] \and Vitali Wachtel\thanks{Faculty of Mathematics, Bielefeld University, Bielefeld, Germany. \\
D. Denisov and W. FitzGerald were supported by a Leverhulme Trust Research Project Grant, RPG-2021-105.
V. Wachtel was funded by the Deutsche Forschungsgemeinschaft (DFG, German Research
 Foundation)—Project ID 317210226—SFB 1283.\\
Email addresses: denis.denisov@manchester.ac.uk, william.fitzgerald@manchester.ac.uk, 
 wachtel@math-uni.bielefeld.de.
 }}
\date{}

\begin{document}
\maketitle
\begin{abstract}
The Airy line ensemble is a random collection of continuous ordered paths that plays an important role within random matrix theory and the Kardar-Parisi-Zhang universality class. The aim of this paper is to prove a universality property of the Airy line ensemble. We study growing numbers of i.i.d.~continuous-time random walks which are then conditioned to stay in the same order for all time using a Doob $h$-transform. We consider a general class of increment distributions; a sufficient condition is the existence of an exponential moment and a log-concave density. We prove that the top particles in this system converge in an edge scaling limit to the Airy line ensemble in a regime where the number of random walks is required to grow slower than a certain power (with a non-optimal exponent 3/50) of the expected number of random walk steps.
Furthermore, in a similar regime we prove that the law of large numbers and fluctuations of linear statistics agree with non-intersecting Brownian motions.
\end{abstract}


\paragraph{2020 MSC classes.} Primary 60G50, 60K35; secondary 60G40, 60F17.

\section{Introduction}
The Airy line ensemble is a random collection of continuous ordered paths; it plays an important role in the \emph{Kardar-Parisi-Zhang} (KPZ) universality class due to its relationship to the (conjecturally) universal limiting objects, the directed landscape and the KPZ fixed point. 
%$\mathcal A_1 > \mathcal A_2 > \ldots$. 
The Airy line ensemble was first considered in Pr\"{a}hofer and Spohn~\cite{ps2002} where they proved convergence in a suitable scaling limit of the finite dimensional distributions of the multi-layer polynuclear growth model. The marginal distribution of the top line at a single time is governed by the Tracy-Widom GUE distribution from random matrix theory~\cite{BDJ}. 
The type of convergence to the Airy line ensemble was then extended by Corwin and Hammond~\cite{corwin_hammond} in the setting of non-intersecting Brownian bridges to a 
functional limit theorem. This established that the paths in the Airy line ensemble are locally absolutely continuous with respect to Brownian motion and remain in the same order for all time. 
The Airy line ensemble plays an important role in the construction of the directed landscape, a universal limit object constructed in Dauvergne, Ortmann and Vir\'{a}g~\cite{dov}. Moreover, the KPZ fixed point, constructed in Matetski, Quastel and Remenik \cite{MQR}, is then related to the directed landscape by a variational formula.
These relationships motivate the study of universality properties of the Airy line ensemble.

We consider universality properties for the Airy line ensemble in the context of continuous-time
random walks, for a general class of increment distributions, which are conditioned 
to stay ordered using $h$-transforms. The resulting $h$-transform has been given a variety of different names: \emph{ordered random walks}, \emph{non-colliding processes}, \emph{vicious random walks} or \emph{non-intersecting paths}. Nearest-neighbour ordered random walks provide an example where exact formulas are present, for a survey see \cite{konig2005}, however, in the general setting there are no explicit exact formulas. 
For a general class of increment distributions and in a discrete-time setting, ordered random walks were first constructed in Eichelsbacher and K\"{o}nig~\cite{eichelsbacher_konig} under non-optimal moment conditions and then with optimal moment conditions in Denisov and Wachtel~\cite{denisov_wachtel10}.
Ordered random walks are an example of random processes evolving under various types of constraints. Many different processes of this type have been studied, for example 
\cite{persistence_survey, denisov_wachtel15, fayolle}, often motivated by relationships to queueing networks, statistical mechanics or combinatorics.

The Airy line ensemble has been studied recently from a variety of different perspectives. The top particle in the Airy line ensemble is the $\text{Airy}_2$ process; we refer to~\cite{remenik_quastel} for a review of the Airy processes and discuss here papers which treat the whole line ensemble. One focus has been on developing characterisations of line ensembles with the Brownian Gibbs property in terms of the behaviour of the top path. The strongest result of this type is~\cite{aggarwal_huang2} following earlier work and conjectures in~\cite{corwin_hammond, corwin_sun, dimitrov_matetski}. These characterisations are expected to be useful in proving that discrete models converge to the Airy line ensemble even in settings where the discrete model may fail to have the Markov property prior to taking the scaling limit. 
Other recent themes have been the development of quantifiable descriptions of the locally Brownian nature of the Airy line ensemble~\cite{hammond} or understanding the spacings between particles in~\cite{dv}.
Variations of the Airy line ensemble have been studied, for example where the ordering condition is weakened to a repulsive interaction in the KPZ line ensemble~\cite{corwin_hammond2}. 
Functional convergence of some \emph{integrable} examples of ordered random walks to the Airy line ensemble 
was proved in~\cite{dnv}.

%
%Brownian last passage percolation is related to the top particle in a collection of independent Brownian motions conditioned using a Doob $h$-transform to stay in the same order for all time. 
%A variety of interpretations of this identity, often having their root in the 
%Robinson-Schensted-Knuth Correspondence, are given in \cite{baryshnikov, gravner_tracy_widom, o_connell2002, warren}. 

%Aside from their relationship to the KPZ universality class, ordered random walks have been used
%in the physics literature, for example to study fibrous structures by de Gennes. 
% and for a variety of further applications discussed in Fisher \cite{fisher}.

In a different direction, another important interpretation of non-intersecting Brownian motions comes from random matrix theory. 
Dyson studied the process given by the real eigenvalues of a Hermitian Brownian motion. 
The evolution is given by a system of stochastic differential equations, named Dyson Brownian motion 
(here $\beta = 2$), which agrees
with independent Brownian motions conditioned not to collide using an $h$-transform, see~\cite{grabiner}. 
We refer to this process as non-intersecting Brownian motions to distinguish from Dyson Brownian motion with general $\beta$.
Universality questions have been considered within random matrix theory over many years, for example see \cite{erdos_yau, johansson_universality, tao_vu} and the references within.
The results that are of most relevance to this setting are those relating to the Deformed Gaussian Unitary Ensemble that corresponds to non-intersecting Brownian motions started from general initial conditions. 
A variety of universality properties in the edge scaling limit are known for this model~\cite{peche, Shcherbina1} that include large classes of initial conditions but with the restriction of only considering the marginal distribution at a fixed time. Convergence of finite dimensional distributions is shown in~\cite{claeys} 
which shows that the Airy line ensemble can also appear in the interior of the spectrum. 

%
%Ordered random walks can be separated into the with nearest-neighbour jumps provide two examples that are referred to as \emph{exactly solvable}. Focusing on the Brownian case, the 
%Vandermonde determinant is a harmonic function for independent Brownian motions killed if they become disordered. The transition density of the killed process is given by the Karlin McGregor formula. 
%These two facts allow the distribution of many statistics of interest to be expressed in terms of determinants or Fredholm determinants. These formulas are tractable for asymptotics 
%in the limit as the number of non-colliding Brownian motions grows to infinity. 
%Beyond the exactly solvable setting, ordered random walks were considered in \cite{eichelsbacher_konig, 
%denisov_wachtel10}. 
%The harmonic function was first constructed in \cite{eichelsbacher_konig} and then under optimal moment conditions in \cite{denisov_wachtel10}. \\
%Convergence to Dyson Brownian motion at fixed $d$.

The existing results which are closest to the present paper are 
universality properties of last passage percolation in `thin' rectangles considered in Bodineau and Martin~\cite{bodineau_martin}
and Baik and Suidan~\cite{baik_suidan1}; and universality properties of ordered random walk bridges in Baik and Suidan~\cite{baik_suidan2}. 
Define $T(n, k)$ to be the last passage percolation time at the point $(n, k) \in \mathbb{Z}_+^2$ constructed from a random environment made up of i.i.d.~random variables for a large class of distributions. 
Then~\cite{baik_suidan1, bodineau_martin} show that $T(n, \lfloor n^a \rfloor)$ converges in a suitable scaling limit to the Tracy Widom GUE distribution whenever $a < 3/7$. The main technique in the proof is an application of the Koml\'{o}s-Major-Tusn\'{a}dy (KMT) coupling of random walks and Brownian motions.
Baik and Suidan applied the same technique to study a slowly growing number $d$ of non-intersecting random walk bridges for a general class of increment distributions. The bridge condition is imposed at time $0$ and $2$, there are $N$ random walk steps and the positions are evaluated at time $1$. They showed convergence of the top particle in a scaling limit to the Tracy Widom GUE distribution. The application of the KMT coupling is less accurate in this setting due to the need to account for the small spacings between particles.
As a result, the convergence in~\cite{baik_suidan2} requires the condition $N \geq (2d)^{8d^3 + 4d^2}$. 

The aim of our paper is to study a growing number $d$ of continuous-time random walks conditioned to stay ordered for a general class of increment distributions and to prove functional convergence of the top particles in an edge scaling limit to the Airy line ensemble. We require the number of ordered random walks to grow sufficiently slowly in the sense that
$d \leq T^{a}$ where $T$ is a parameter appearing in the edge scaling limit corresponding to the expected number of random walk steps taken.
One of the main ideas in the proof is an analysis of the harmonic function for ordered random walks. In particular, we construct a superharmonic function by using properties of likelihood ratio 
ordering that can be used to give uniform estimates in $d, T$. 
This immediately gives an upper bound on the harmonic function and plays a useful role in establishing a lower bound. 
This understanding of the harmonic function can be combined with couplings and estimates on how quickly ordered random walks separate as well as their behaviour up until the time that they have separated in order to establish convergence in an edge scaling limit to the Airy line ensemble. We also consider the law of large numbers and fluctuations of linear statistics 
and prove that these agree with non-intersecting Brownian motions in a regime $d \leq T^{a'}$ with a different exponent $a'.$
The exponents $a, a'$ in our results are not optimal and 
we discuss possible ways to improve this. 

In Section~\ref{results} we state our main results. In Section~\ref{sec:superharmonic} we construct a superharmonic function for ordered random walks. In Section~\ref{sec:lower_bound} we prove a lower bound for the harmonic function of ordered random walks. In Section~\ref{sec:convergence} we prove convergence of ordered random walks to the Airy line ensemble in an edge scaling limit as well as showing that the law of large numbers and fluctuations of linear statistics agree with non-intersecting Brownian motions. In the Appendix, we provide further details on likelihood ratio ordering which may be convenient for the reader.

%Baik and Suidan proved that the one-point distribution of slowly growing numbers of non-intersecting random walk bridges converges in a certain scaling limit to the Tracy Widom GUE distribution. This used two main ideas (i) the Komlos-Major-Tusnady coupling of random walks and Brownian motions to control the errors between random walk bridges and Brownian bridges started away from zero and (ii) to prove that the exact solvability of non-colliding Brownian bridges extends to general initial conditions away from zero. One challenge is the need to control errors in the coupling for a growing number of particles and this led to the constraint on 
%$d, n$ that $n \geq (2d)^{8d^3 + 4d^2}$.
%The constraint on $d, n$ can be compared to work of Bodineau, Martin and Baik, Suidan on last passage percolation in thin regions. In that case an analogous KMT coupling leads to a significantly weaker constraint on the relationship between $d$ and $n$ as $d \leq n^{b}$ 
%for $b < 3/7$. 

\section{Statement of results} 
\label{results}

\subsection{Background and notation for ordered random walks}
\label{sec:notation}
Let $(e_{ij})_{i \geq 1, 1 \leq j \leq d}$ be a collection of i.i.d.~exponential random variables with rate $1$ and
let $(X_{ij})_{i \geq 1, 1 \leq j \leq d}$ be a collection of
i.i.d.~random variables. 
Let $x = (x_1, \ldots, x_d)$ and $(S(t))_{t \geq 0} = (S_1(t), \ldots, S_d(t))_{t \geq 0}$.
For each $j = 1, \ldots, d$ and $t \geq 0$ let 
\[
S_j(t) = x_j + \sum_{i \geq 1} X_{ij} 1_{\{\sum_{k=1}^i e_{k j} \leq t\}}.
\]
Let $W^d = \{(x_1, \ldots, x_d) \in \mathbb{R}^d: x_1 < x_2 < \ldots < x_d\}$ denote the Weyl chamber. 
% {\color{blue}SUGGESTION: It would be good to use strict inequalities in the definition of the Weyl chamber, as they are more standard.} 
%We define a $d$-dimensional random walk $(S(n))_{n \geq 0} = (S_1(n), \ldots, S_d(n))_{n \geq 0}$ started from $S(0) = x = (x_1, \ldots, x_d) \in W^d$ by 
%$S_j(n) = x_j + \sum_{i=1}^n X_{ij}$ for $n \geq 1$ and $j = 1, \ldots, d$. 
Define the stopping time 
\[ 
\tau  = \inf\{t \geq 0:  S(t) \notin W^d\}.
\]
Let $d > 3$. Suppose that $\E(X_{ij}) = 0$ (which can be achieved by centering) and $\E(X_{ij}^{d-1}) < \infty$. 
%first in~\cite{eichelsbacher_konig} (with stronger moment conditions) and then in~\cite{denisov_wachtel10} with optimal moment conditions 
There exists a
harmonic function $V: W^d \rightarrow \mathbb{R}_{> 0}$ such that for all $t > 0$,
\[
\E_x[V(S(t)); \tau>t]= V(x) 
\]
which is strictly positive and integrable on $W^d$.
This function is given by 
\begin{equation}
\label{V_defn}
  V(x) = \Delta(x)-\E_x[\Delta(S(\tau))] = \lim_{t \rightarrow \infty} \E_x[\Delta(S(t)) 1_{\{\tau > t\}}]
\end{equation}
where \[
\Delta(x) = \prod_{1 \leq i < j \leq d}(x_j - x_i)
\] is the Vandermonde determinant. 
A direct computation shows that $V$ satisfies the harmonic condition and the main difficulty is to show that $V$ is finite and $V(x) > 0$ for all $x \in W^d$.
This was shown in~\cite{denisov_wachtel10, eichelsbacher_konig} in the discrete-time case; the continuous-time setting here can be proved in a similar way. Alternatively, we can use the methods developed in this paper: $V$ is integrable by Proposition~\ref{superharmonic} and $V(x) \sim \Delta(x)$ as $x_{j+1}-x_j \rightarrow \infty$ by Proposition~\ref{lem:asymptotics.v} which is sufficient to show that $V(x) > 0$ for all $x \in W^d$, see 
the proof of Proposition 4e in~\cite{denisov_wachtel10}. 
Once these conditions have been established $V$ can be used to construct
a process in $W^d$ by a Doob $h$-transform. We refer to this process as an 
\emph{ordered random walk} and let $\mathbf{E}^{(V)}$ denote the expectation under this change of measure: for bounded measurable $\phi : \mathbb{R}^n \rightarrow, \mathbb{R}$
\[
\E^{(V)}_x[\phi(S(t))] = \E_x\left[\phi(S(t)) \frac{V(S(t))}{V(x)}; \tau > t\right].
\]
Note that the evolution of this process only involves one co-ordinate jumping at a time.
It was shown in~\cite{denisov_wachtel10, eichelsbacher_konig} (in a discrete-time setting) that this $h$-transform agrees with the process that is conditioned to stay in $W^d$ up to a time $t$ followed by taking the limit $t \rightarrow \infty$. 
Moreover,~\cite{denisov_wachtel10, eichelsbacher_konig} show that for fixed $d$, ordered random walks converge weakly to
non-intersecting Brownian motions. We consider the case of
growing numbers of random walks when 
$d=d(n)$ grows to infinity. 

Non-intersecting Brownian motions can be defined in an analogous way. 
Let $B_1, \ldots, B_d$ be independent Brownian motions started from $x_1, \ldots, x_d$ respectively and let 
$(B(t))_{t \geq 0} = (B_1(t), \ldots, B_d(t))_{t \geq 0}.$
Define $\tau^{\text{bm}} = \inf\{t \geq 0: B(t) \notin W^d\}.$
The Vandermonde determinant $\Delta$ satisfies $\E_x[\Delta(B(t)); \tau^{\text{bm}} > t]
= \Delta(x)$ and non-intersecting Brownian motions can then be constructed as a Doob $h$-transform. For bounded measurable $\phi : \mathbb{R}^n \rightarrow \mathbb{R}$ let
\[
\E_x^{(\Delta)}[\phi(B(t))] = \E_x\left[\phi(B(t)) \frac{\Delta(B(t))}{\Delta(x)}; \tau^{\text{bm}} > t\right].
\]
We refer to \cite{konig2002} for a survey of the above and the definition of the entrance law for non-intersecting Brownian motions started from zero.

%
%As $d$ grows it is convenient to apply a rescaling and to formulate the co-ordinates at a fixed time as forming a point process.
%The measure corresponding to the initial condition is given by \[
%\mu_0^{(n)} = \frac{1}{d_n} \sum_{i=1}^{d_n} \delta_{x_i/\sqrt{nd_n}}.
%\]
%For each $k \geq 1$,  \[
%\mu_{k}^{(n)} =  \frac{1}{d_n} \sum_{i=1}^{d_n} \delta_{S_i(k)/\sqrt{kd_n}}.
%\]
%%Let $\rho^{(m)}_{k, n}$ be the $m$-point joint intensities of $\mu_{k}^{(n)}$ for each $k \geq 0$.
%%The one-point intensity will be denoted $\rho_{k, n} = \rho^{(1)}_{k, n}$ and
%%$\rho^{(m)}_n = \rho^{(m)}_{n, n}$.
%The Stieltjes transform of a measure $\nu$ is denoted by
%\[
%f_{\nu}(z) = \int_I \frac{\nu(dt)}{z-t}, \qquad z \in \mathbb{C} \setminus I.
%\]
%
%Convergence to limiting measure 
%Bulk
%
%For our approximations between random walks and Brownian motion it is convenient to consider the set where the random walk co-ordinates are separated
%\[
%W_{n,\varepsilon} =\{x\in W \colon 
%x_j-x_{j-1}\ge n^{1/2-\varepsilon}, j=2,\ldots,d \}.  
%\]

\subsection{Class of random walks considered.} 
\label{sec:rw_assumptions}
We suppose that $\E(X_{ij}) = 0$ and $\text{Var}(X_{ij}) = 1$.
The simplest to state conditions under which our results hold are the following:
\begin{enumerate}
\item There exists $\delta_0 > 0$ such that $\E(e^{\delta X_{ij}}) < \infty$ for all $\delta \in (-\delta_0, \delta_0)$.
\item $X_{ij}$ has a log-concave density $f$, that is 
\[
f(\theta x +(1-\theta)y)\ge f(x)^\theta f(y)^{1-\theta}, \quad x,y\in \mathbb R, \theta \in [0,1].  
\]
\end{enumerate}
In Section~\ref{sec:superharmonic} we replace condition (ii) by a more general but harder to state condition, 
where a density is not needed. 
Many common distributions satisfy conditions (i) and (ii).

% eg.~normal, exponential, Laplace, gamma with shape $\geq 1$ and beta with both shape parameters $\geq 1$.

%
%\subsection{Process level approximation of ordered random walks and Dyson Brownian motion}
%\begin{theorem}[Process level approximation of ordered random walks and Dyson Brownian motion]
%Let $f : W^d \rightarrow \mathbb{R}$ be a uniformly continuous function
%
%\end{theorem}

%
%\subsection{Approximation of the harmonic function} 

\subsection{Convergence of ordered random walks to the Airy line ensemble}

The parabolic Airy line ensemble arises as the edge scaling limit of non-intersecting Brownian motions. 
It can be defined as a sequence $\L= (\L_1, \L_2, \ldots)$ of random continuous functions which are ordered as $\L_1(t) > \L_2(t) > \ldots$ for all $t \in \mathbb{R}$ and with a specified determinantal formula for each $t_1 < t_2 < \ldots < t_k$ for the distribution of
$\{\L_i(t_j) : i \in \mathbb{N}, j \in \{1, \ldots, k\}\}.$ The exact distribution is not needed here and can be found in~Definition 2.1 of~\cite{dnv}.
The Airy line ensemble is a stationary process related to the parabolic Airy line ensemble as 
$\mathcal{A}(t) = \L(t) + t^2$.

Define a collection of non-intersecting Brownian motions ordered as $B_1 \leq \ldots \leq B_N$ all started from zero. 
Let 
\begin{align}
\label{edge_scaling_limit}
D_i^N(t) = 
N^{1/6}\left((B_{N-i+1}(1 + 2N^{-1/3 }t) - 2N^{1/2}
- 2N^{1/6}t\right).
\end{align}
% = \frac{N^{1/6}}{2^{1/2}}(1-N^{-1/3 }t)B_i\left(\frac{1 + N^{-1/3 }t}{1 - N^{-1/3 }t}\right) - 2^{1/2}N
% \]
% bridges with diffusion parameter $2$ by $(W_1, \ldots, W_N)$ with $W_i(-N) = W_i(N) = 0$.
% For $t \in [-N^{1/3}, N^{1/3}]$ let $D_i^N(t) = N^{-1/3}(W_i(N^{2/3} t) - 2^{1/2} N)$. 
Let $D^N(k, A)$ denote the set of the top $k$ lines at times in $A$.
For any $k \geq 1$ and $L > 0,$ Corwin and Hammond~\cite{corwin_hammond} showed that the line ensemble $D^N(k, [-L, L])$ converges weakly as $N \rightarrow \infty$ to $\L$ restricted to $\{1, \ldots, k\} \times [-L, L].$
This was proved for non-intersecting Brownian bridges in Theorem 3.1 of 
\cite{corwin_hammond} (note the definition of the parabolic Airy line ensemble differs by a factor of $2^{1/2}$). This can be converted into a statement about non-intersecting Brownian motions by the usual transformations between Brownian bridges and Brownian motions. The convergence in the form stated here can be found in Theorem 2.1 of \cite{dov}. Note that the Airy line ensemble is absolutely continuous with respect to Brownian motions with diffusivity parameter $2$ meaning that the quadratic variation grows as $2(t-s)$ in an 
interval $[s, t].$ The original Brownian motions have diffusivity parameter $1$ but this is transformed to remove factors of $2^{1/2}$ in the scaling limit.

%For $t \in [-N^{1/3}, N^{1/3}]$,
%\[
%\left(\left( \frac{N^{1/6}}{2^{1/2}}(1-N^{-1/3 }t)B_i\left(\frac{1 + N^{-1/3 }t}{1 - N^{-1/3 }t}\right) - 2^{1/2}N \right) \right)_{i=1}^k 
%\]
%converges weakly to $\L$ restricted to $\{1, \ldots, k\} \times [-T, T]$.
%
%For $t \in [-T, T]$ can approximate 
%\[
%\frac{1 + N^{-1/3 }t}{1 - N^{-1/3 }t} = (1 + 2N^{-1/3}t + O(N^{-2/3}))
%\]
% This can be stated in terms of non-intersecting Brownian motions with diffusion parameter $2$ as \[
% \frac{N^{1/6}}{2^{1/2}} \left( B_i(1+2N^{-1/3} t) - 2 N^{1/2}\right) 
% \]
% converges weakly to $\L$ restricted to $\{1, \ldots, k\} \times [-L, L]$.
%Non-intersecting Brownian motions have Brownian scaling relations, which can be seen from their matrix representation, so
Furthermore let $N = \lfloor T^a \rfloor$ (we omit $\lfloor \cdot \rfloor$ to simplify the notation) and use Brownian scaling to rewrite \eqref{edge_scaling_limit} as
\begin{align}
T^{a/6-1/2} \left(B_{N-i+1}(T+2T^{1-a/3} t) - 2 T^{1/2+a/2}-2T^{1/2+a/6}t\right) 
%\left( \frac{n^{a/6}}{2^{1/2}} \left(B_i(1+2n^{-a/3} t) - 2 n^{a/3}\right) \right)_{i=1}^k
\end{align}
%For random walks for $t \in [-T, T]$, define
%\[
%X_{i}(t) =  \frac{n^{a/6}}{2^{1/2}}\left( \frac{S_i([n+2n^{1-a/3} t])}{n^{1/2}} - 2 n^{a/2} \right).
%\]
This now motivates the following definition of the edge scaling limit for ordered random walks. 
Define for $t \in [-L, L]$ and $i = 1, \ldots, k$,
\[
X_{i}^T(t) = T^{a/6-1/2}\left( S_{d-i+1}(T+2T^{1-a/3} t) - 2 T^{1/2+a/2} 
- 2T^{1/2 + a/6} t\right).
\]
The conventions in the existing literature are that the ordered random walks are ordered as $S_1(t) \leq \ldots \leq S_d(t)$ while in the Airy line ensemble $\L_1(t) > \L_2(t) > \ldots.$ We stick with both of these conventions at the expense of needing to define $X$ in terms of $S$ in a reversed order.

Let $X^T(k, A)$ denote the top $k$ lines at times in $A \subset \mathbb{R}$
as an element of $D(k, A)$, the set of $k$-tuples of c\`{a}dl\`{a}g functions on $A$ with the $\sigma$-algebra generated by the finite-dimensional evaluation maps.
We consider the uniform metric on $D(k, A)$,
\begin{equation}
    \label{rho_eqn}
    \rho(x, y) = \sup_{1 \leq j \leq k} \sup_{t \in A} \lvert x_k(t) - y_k(t) \rvert
\end{equation}
for $x = (x_1(t), \ldots, x_k(t))_{t \in A}, y = (y_1(t), \ldots, y_k(t))_{t \in A} \in D(k, A)$.   
We consider weak convergence in this setting, in the sense of Section 6 and Section 15 of \cite{billingsley}.
Alternatively, we could restrict $X^T(k, A)$ by taking a linear interpolation to the set $C(k, A)$ of $k$-tuples of continuous functions on $A$, with the uniform topology, and weak convergence would follow in this setting after some adaptions to
Proposition \ref{prop.coupling}. 

%Let $\Sigma$ be an interval of $\mathbb{Z}$ and $\Lambda$ an interval of $\mathbb{R}$. 
%Let $X$ be the set of continuous functions $f : \sigma \times \Lambda \rightarrow \mathbb{R}$ endowed 
%with the topology of uniform convergence on compact subsets of $\Sigma \times \Lambda$. 
Our first result is the following.

\begin{theorem}
\label{thm:airy}
Suppose the conditions in Section~\ref{sec:rw_assumptions} hold.
%Let $f$ be a bounded continuous functional from $X$ to $\mathbb{R}$. 
For $d = \lfloor T^a \rfloor$ and $a < 3/50$ and all starting points $\lvert x_j \rvert = o(T^{1/2-a/6})$  
the line ensemble $X^T(k, [-L, L])$ converges weakly as $T \rightarrow \infty$ to $\L$ restricted to $\{1, \ldots, k\} \times [-L, L].$
%\[
%\E[f(X_1^n, \ldots, X_k^n)] \rightarrow \E[f(\L)], \quad n \rightarrow \infty. 
%\]
\end{theorem}

We state a corollary involving convergence to the Tracy Widom GUE distribution. 
\begin{corollary}
With the assumptions of Theorem \ref{thm:airy}, 
\[
T^{a/6-1/2} \left( S_d(T)-2T^{1/2+a/2}\right) \stackrel{d}{\rightarrow} F_2
\]
where $F_2$ is the Tracy-Widom GUE distribution.
\end{corollary}

%The Theorem relies on two steps. One step that was completed in Corwin and Hammond Theorem 3.1 is that $D^N(k, [-T, T])$ converges weakly as $N \rightarrow \infty$ to $\L$ restricted to $\{1, \ldots, k\} \times [-T, T].$ 
As stated above, this convergence is known for non-intersecting Brownian motions started from an entrance law at zero by~\cite{corwin_hammond}, see Theorem 2.1 of \cite{dov} for the exact statement we use.
The contribution of our paper is an approximation of ordered random walks by non-intersecting Brownian motions
that is valid for growing $d$. 
The main ingredient is the following approximation for the harmonic function $V$ given in~\eqref{V_defn}.
For any $\veps>0$, define 
\[
	W_{T,\veps}=\left\{x\in \mathbb{R}^d:x_{j+1}-x_{j}>T^{1/2-\veps},1\le j< d\right\}.
\]
\begin{proposition}\label{lem:asymptotics.v}
Suppose the conditions in Section~\ref{sec:rw_assumptions} hold, $a\in (0, 1/4)$ and $\varepsilon > 0$.
  Uniformly in $x\in W_{T,\varepsilon}$ and $d\le T^{a}$ with $2a < 1/2 - \veps$, 
  \[
  V(x)  \sim \Delta(x). 
  \]
\end{proposition}

The upper bound follows from the construction of a supermartingale in Section~\ref{sec:superharmonic}. 
The lower bound is then established in Section~\ref{sec:lower_bound}.

Our focus is on universality with respect to the law of the random walks where the most challenging case is when all particles are started from zero. 
  The arguments we give to approximate functionals of ordered random walks 
  by non-intersecting Brownian motions hold more generally, including when the limiting object is not the Airy line ensemble. In Section \ref{sec:LLN} we state two results involving macroscopic limits. Equally our methods could also be applied to other microscopic limits such as the Dyson sine process, non-intersecting Brownian motions with outliers or the Pearcey process. 
 It is worth noting that the convergence of non-intersecting Brownian motions to a limiting object such as the Airy line ensemble is only known in particular cases, 
 e.g. for some special choices of initial conditions or in the sense of finite-dimensional distributions.

The asymptotics in Proposition~\ref{lem:asymptotics.v}
could also be combined with papers~\cite{peche, claeys, Shcherbina1} that prove convergence in the edge scaling limit of the one-point or finite-dimensional distributions of non-intersecting Brownian motions to the Airy line ensemble for a more general class of initial conditions.
 If we have $x \in W_{T, \veps}$ in addition to the conditions for Theorem~\ref{thm:airy}  then we can improve the value of $a$. The second part of the proof of Theorem~\ref{thm:airy} and thus the conditions for Lemma~\ref{lem:error_bounds3} are not needed. If we wanted to relax the condition $\lvert x_j \rvert = o(T^{1/2-a/6})$ then we could combine the 
approximation arguments used here with the results of~\cite{peche, claeys, Shcherbina1}.
Note that~\cite{peche, Shcherbina1} apply only at the level of marginal distributions rather than at the process level. The results of~\cite{claeys} apply to finite-dimensional distributions which can then be combined with characterisations of the Airy line ensemble in terms of its finite dimensional distributions, e.g.~\cite{dimitrov_matetski}.

\subsection{Law of large numbers and fluctuations of linear statistics}
\label{sec:LLN}

We now consider macroscopic properties of the system of ordered random walks. In particular, we show that the law of large numbers and fluctuations of linear statistics agree with the case of non-intersecting Brownian motions. 
Let $D([\delta,1])$ denote the set of c\`{a}dl\`{a}g functions with the sigma-algebra generated by the finite-dimensional evaluation maps and the uniform metric
$\rho$ given in \eqref{rho_eqn}.
Let $\delta > 0$ and define 
for $f:D([\delta,1])\to \mathbb [0, 1]$,
\[
X_T(f) = \sum_{j=1}^d f(T^{-1/2-a/2} S_j(Tt) : t \in [\delta, 1])
\]
and 
\[
Y_T(f) = \sum_{j=1}^d f(T^{-1/2-a/2} B_j(Tt) : t \in [\delta, 1]).
\]

\begin{theorem}
\label{thm:brownian_approx}
Suppose the conditions in Section~\ref{sec:rw_assumptions} hold.
Let $\delta > 0$ and suppose that $f:D([\delta,1])\to \mathbb{R}$ is measurable, uniformly continuous
and $g : \mathbb{R} \rightarrow [0, 1]$ is uniformly continuous.
%Let $f$ be a bounded continuous functional from $X$ to $\mathbb{R}$. 
For $d = \lfloor T^a \rfloor$ with $a < 1/16$ and all starting positions $\lvert x_j \rvert = o(T^{1/2+a/2})$,
\[
\E_x^{(V)}[g\left(d^{-1}X_T(f)\right)] - \E_x^{(\Delta)}[g\left(d^{-1}Y_T(f)\right)] \rightarrow 0, \quad T \rightarrow \infty. 
\]
\end{theorem}

This characterises the distribution of $d^{-1}X_T(f)$ under $\mathbf{P}_x^{(V)}$.
If $f$ only depends on the path at time $1$ then the empirical measure of non-intersecting Brownian motions 
started from zero converges to the Wigner semicircle law almost surely.  
The evolution of the empirical measure in time is deterministic and has been studied by~\cite{rogers_shi}. Theorem~\ref{thm:brownian_approx} extends these results to the setting of ordered random walks. As in the case of the Airy line ensemble we focus on the case where the random walks are all started close to zero as this is the most difficult approximation argument.

The next property we consider is the fluctuations of linear statistics.  

\begin{theorem}
\label{thm:brownian_approx_clt}
Suppose the conditions in Section~\ref{sec:rw_assumptions} hold.
Let $\delta > 0$ and suppose that $f : D[\delta, 1] \rightarrow \mathbb{R}$ is measurable, uniformly continuous and $g : \mathbb{R} \rightarrow [0, 1]$ is uniformly continuous.
%Let $f$ be a bounded continuous functional from $X$ to $\mathbb{R}$. 
For $d = \lfloor T^a \rfloor$ with $a < 1/16$ and all starting positions $\lvert x_j \rvert = o(T^{1/2+a/2})$,
\[
\E_x^{(V)}\left[g\left(X_T(f) - \E^{(V)}_x [X_T(f)]\right)\right] - \E_x^{(\Delta)}\left[g\left(Y_T(f) - \E_x^{(\Delta)} [Y_T(f)]\right)\right] \rightarrow 0
\]
as $T \rightarrow \infty$.
\end{theorem}

At a fixed time, the fluctuations of linear statistics of eigenvalues of random matrices have been considered in many papers starting with the work of~\cite{diaconis, johansson}.
For processes evolving in time, the fluctuations of linear statistics have been studied by Duits~\cite{duits} for non-intersecting Brownian bridges where the 
function $f : C[0, 1] \rightarrow \mathbb{R}$ has the form
$f(w) = \sum_{i=1}^N f_i(w(t_i))$ 
for $0 < t_1 < \ldots < t_N < 1$ where $x \rightarrow f_i(x)$ for all $i = 1,\ldots, N$ is continuously differentiable and grows at most polynomially for $x \rightarrow \pm\infty$. 
After applying time changes to the setting of non-intersecting Brownian motions, the results of~\cite{duits} show for this class of functions that
\[
Y_T(f) - \E_x^{(\Delta)} [Y_T(f)] \rightarrow N(0, \sigma^2_f), \quad T \rightarrow \infty
\]
where $\sigma^2_f$ in an explicit variance. The remarkable feature of central limit theorems of this type is that there is no normalisation factor in $T$. We refer to~\cite{duits} for further discussion and interpretations in terms of the Gaussian free field.

\begin{remark}
All of the Theorems in Section \ref{results} hold with Condition (ii) in Section \ref{sec:rw_assumptions} replaced by Condition (ii)' in Section \ref{sec:superharmonic}.
\end{remark}

\section{Superharmonic functions}
\label{sec:superharmonic}

Let $U, W$ be real-valued random variables. We say that $U$ is smaller than $W$ in likelihood ratio order 
denoted $U \stackrel{lr}{\leq} W$ if 
\[
\pr(U \in A) \pr(W \in B) \geq \pr(U \in B) \pr(W \in A)
\]
for all measurable subsets $A, B$ such that $\sup\{x : x \in A\} \leq \inf\{y: y \in B\}$. 
% if $x \in A$ and $y \in B$ then $x \leq y$. 
When $U$ and $W$ have densities $f_U$ and $f_W$ respectively this can be written as 
\[
f_U(u)f_W(w) \geq f_U(w)f_W(u), \qquad \text{ for all } u \leq w.
\]
If $U \stackrel{lr}{\leq} W$ then $U$ is smaller than $W$ in the usual stochastic order, for example see~\cite{shaked} for some properties of stochastic orders.

We can replace Condition (ii) in Section~\ref{sec:rw_assumptions} with the 
following condition that we call Condition (ii)'. 
Let $X$ denote a random variable with the same law as $(X_{ij})_{i \geq 1, 1 \leq j \leq d}$.
Let $\zeta$ be a random variable taking values in $[1, \infty)$ such that there exists $\delta_0 > 0$ such that $\mathbf{E}[e^{\delta \zeta}] < \infty$ for all $\delta < \delta_0.$ If there exists $M > 0$ such that $\mathbf{P}(\lvert X \rvert \geq M) = 0$ then we set $\zeta := M$. Otherwise, suppose that we can find 
$\zeta$ satisfying  such that
for all $\theta \geq 0$,
\begin{align}
\label{defn_zeta1}
(X  - \theta) \vert \{X  > \theta\} & \stackrel{lr}{\leq} \zeta \\
\label{defn_zeta2}
-(X  + \theta) \vert \{X  < -\theta\} &\stackrel{lr}{\leq} \zeta. 
\end{align}
Here, $(X- \theta) \vert \{X > \theta\}$ is the random variable whose law is that of $X-\theta$, conditioned on the event $X > \theta$.
If $X$ is almost surely bounded above (resp. below) by M then we take $\zeta \geq M$ satisfying \eqref{defn_zeta2} (resp. \eqref{defn_zeta1}). 
%Then Theorem~\ref{thm:airy} also holds with condition (ii) replaced by (ii)'.

A sufficient condition on $X$ that allows a simple construction of $\zeta$ is the following Lemma, proved in Appendix~\ref{appendix:lro}. 
Hence, Condition (ii)' implies Condition (ii).
\begin{lemma}
\label{log_concave}
Suppose that $X$ has a density $f$ that is log-concave. Define
\[
\zeta:= (X \vert X > 0) -(X \vert X < 0) + 1.
\] 
Here, $(X \vert X > 0) -(X \vert X < 0)$ is denoting the difference of two independent random variables, whose laws are $X$ conditioned on $X > 0$
and $X < 0$ respectively. Then, $\zeta$ takes values in $[1, \infty),$ has a finite exponential moment and satisfies~\eqref{defn_zeta1} and~\eqref{defn_zeta2}. 
\end{lemma}

Let $\eta_1: = 0$ and for $j = 2, \ldots, d$ let $\eta_j := \sum_{i = 1}^{j-1} \zeta_i$ where $(\zeta_i)_{i =1}^{d-1}$ are i.i.d.~with common distribution $\zeta$ and are also independent of any of the $(X_{ij})_{i \geq 1, 1 \leq j \leq d}$. 
Let $\eta = (\eta_1, \ldots, \eta_{d}) \in W^d$.

\begin{proposition}
\label{superharmonic}
Suppose the conditions in Section \ref{sec:rw_assumptions} hold, where either Condition (ii) or (ii)' can be imposed.
For $x \in W^d,$ let \[
h(x) = \E\left[\prod_{1 \leq i < j \leq d} (x_j - x_i + \eta_{j} - \eta_i)\right].
\]
Then 
%$h$ is a superharmonic function for $(S(n))_{n \geq 0}$ killed when $\tau$ occurs,
\[
\E_x[h(S(t)) 1_{\{\tau > t\}}] \leq h(x), \quad x \in W^d.
\]
Furthermore
\[
V(x) \leq h(x), \quad x \in W^d.
\]
\end{proposition}

The form for the superharmonic function can be compared to the harmonic function found in~\cite{DF} for ordered exponential random walks. 
%In Appendix~\ref{appendix:superharmonic} we prove a version of this result for a random walk where at each time step only one co-ordinate jumps. This simplifies some technicalities in the proof. 

The reason for introducing likelihood ratio orderings is the following characterisation; we provide 
the simple proof in Appendix~\ref{appendix:lro}.
\begin{lemma}
\label{stoch_ineq}
Suppose that $U, W$ are independent positive random variables such that $U \stackrel{lr}{\leq} W$. 
Let $\phi : \mathbb{R}^2 \rightarrow \mathbb{R}$ satisfy (i) $\phi(u, w) + \phi(w, u) \geq 0$ for all $w \geq u \geq 0$ 
and (ii) $\phi(u, w) \geq 0$ for all $w \geq u \geq 0$.
Then, 
\[
\E[\phi(U, W)] \geq 0.
\]
\end{lemma}

\begin{proof}[Proof of Proposition~\ref{superharmonic}]
Let $X_k$ denote a random variable with the same law as $(X_{ij})_{i \geq 1, 1 \leq j \leq d}.$
Suppose that we can show that for $k = 1, \ldots, d-1$,
\begin{equation}
\label{ineq_up}
\E_x[h(x_1, \ldots, x_{k-1}, x_k + X_k, x_{k+1}, \ldots, x_d); x_k + X_k > x_{k+1}] \geq 0
\end{equation}
and for $k = 2, \ldots, d$, 
\begin{equation}
\label{ineq_down}
\E_x[h(x_1, \ldots, x_{k-1}, x_k + X_k, x_{k+1}, \ldots, x_d); x_k + X_k < x_{k-1}] \geq 0.
\end{equation}
Then this means that for any $t \geq 0$,
\begin{equation}
\label{ineq1}
\E_x[h(S(\tau))1_{\{\tau \leq t\}}] \geq 0, \qquad x \in W^d.
\end{equation} 
Here, recall that only one walker can jump at a time.  
Using Theorem 2.1 from~\cite{konig2002} that $\Delta$ is harmonic for $(S(t))_{t \geq 0}$ (here there is no killing)
then $h$ is also harmonic for $(S(t))_{t \geq 0}$ since 
\begin{align*}
\E_x [h(S(t))]  & = \E_x \E^{\eta}[\Delta(S(t)+\eta )] = \E^{\eta} \E_x[\Delta(S(t)+\eta )] \\
& = \E^{\eta}[\Delta(x+\eta)] = h(x).
\end{align*}
Combining this property with~\eqref{ineq1} we have
\begin{align*}
\E_x[h(S(t)) 1_{\{\tau > t\}}] 
& = \E_x[h(S(t) )]  - \E_x[h(S(t)) 1_{\{\tau \leq t\}}] \\
& = h(x)  - \E_x[h(S(\tau)) 1_{\{\tau \leq t\}}]\\
& \leq h(x).
\end{align*}
Therefore $h$ will be superharmonic if we can show~\eqref{ineq_up} and~\eqref{ineq_down}. Using $\Delta(x) \leq h(x)$ for all $x \in W^d$
and the representation from~\eqref{V_defn} for the harmonic function $V$
%\[
%V(x) = \lim_{n \rightarrow \infty} \E_x[\Delta(S(t)) 1_{\{\tau > t\}}]
%\]
then
\begin{align*}
h(x) & \geq \limsup_{t \rightarrow \infty} \E_x[h(S(t)) 1_{\{\tau > t\}}]  \\
& \geq \limsup_{t \rightarrow \infty} \E_x[\Delta(S(t)) 1_{\{\tau > t\}}] \\
& = V(x).
\end{align*}
As a result, it is sufficient to prove~\eqref{ineq_up} and~\eqref{ineq_down}. 
If $X_k$ is almost surely bounded then the equations are automatically satisfied.
For $k = 1, \ldots, d-1$ let $Y = x_k + X_k-x_{k+1} \vert (x_k + X_k > x_{k+1})$ and $Z = \zeta_k$. 
Note that $h$ is an expectation of a product indexed by $(i, j)$ for $1 \leq i < j \leq d.$ We can write this product with the indices in the order $(k, k+1)$, $(j, k)$ for $j < k$, $(j, k+1)$ for $j < k+1$, $(k, j)$ for $j > k+1$, 
$(k+1, j)$ for $j > k+1$ and finally $(i, j)$ for $i < j \notin \{k, k+1\}.$
This gives
\begin{align*}
& \E_x[h(x_1, \ldots, x_{k-1}, x_k + X_k, x_{k+1}, \ldots, x_d)\mid  x_k + X_k > x_{k+1}] \\
& = \E \bigg[ (Z - Y) \prod_{j < k}(Y + x_{k+1} - x_j + \sum_{r=j}^{k-1} \zeta_r) 
\prod_{j < k} (Z + x_{k+1} - x_j+ \sum_{r = j}^{k-1} \zeta_{r}) \\
&\prod_{j > k+1} (x_j - x_{k+1} + Z - Y + \sum_{r=k+1}^{j-1} \zeta_r) 
\prod_{j > k+1} (x_j - x_{k+1} + \sum_{r=k+1}^{j-1} \zeta_r)\\
& \prod_{i < j: i, j \notin\{k, k+1\}}(x_{j} - x_i + \sum_{r = i, r \neq k}^{j-1} \zeta_r + Z 1_{\{i < k, k+1 < j\}})  \bigg].
\end{align*}
Any empty products should be interpreted as having the value $1$.
%Here $W = \prod_{i < j \in P_k} (x_{j} - x_i + \sum_{r = i}^{j-1} \zeta_r) $
%where $P_k = \{(i, j) : i \neq k \text{ and } j \notin \{k, k+1\}\}$ are all of the remaining indices not contained above. 
%Note that $W$ depends on $\zeta_{k}$ but does not depend on $S_k$. Moreover, W is a positive and increasing function of $\zeta_{k}$.
%
%Let $Y = \zeta_{k}$ and $X = (S_{k} - x_{k+1}) \vert S_k > x_{k+1}$ and note that both are positive.
%Then 
%\begin{align*}
%& \E_x(h(S(1))1_{\{\tau = \tau_k^+ = 1\}})\\
%&\E \bigg[(Y - X) \prod_{j < k}(X + x_{k+1}  - x_j + \sum_{r=j}^{k-1} \zeta_r) 
%\prod_{j > k} (x_j -x_{k+1} - X + Y + \sum_{r=k}^{j-1} \zeta_r) \\
%&
%\qquad \prod_{j < k}(Y + x_{k+1} - x_j + \sum_{r = j}^{k-1} \zeta_r) 
%W   \bigg].
%\end{align*}
This can be written as $\E[\sum_{p, q \in \mathbb{Z}_{\geq 0}} \phi_{p, q}(Y, Z)]$ where 
\[
\phi_{p, q}(y, z) = c_{p, q} (z -y)^p z^q  \psi(y, z)
\]
where $c_{p, q} \geq 0$ is a constant and $\psi : \mathbb{R}_+^2 \rightarrow \mathbb{R}_+$ is symmetric and positive. 
The symmetric function $\psi$ arises from the symmetry between the factors in the product for $h$ indexed by $(j, k)$ for $j < k$ and $(j, k+1)$ for $j < k$.
For any $p, q \in \mathbb{Z}_{\geq 0}$ we have that $\phi_{p, q}(y, z) \geq 0$ for all $z \geq y \geq 0$ as each factor is positive. 
Moreover, 
\begin{align*}
 \phi_{p, q}(y, z) + \phi_{p, q}(z, y)  
=  c_{p, q} (z -y)^p\psi(y, z) \big(z^q + (-1)^p y^q \big) 
 \geq 0
\end{align*}
 for all $z \geq y \geq 0$. 
Therefore $\phi_{p, q}$ satisfies the conditions for Lemma~\ref{stoch_ineq}, the random variables
$Y$ and $Z$ are independent with $Z \geq_{lr} Y$ and so $\E[\phi_{p, q}(Y, Z)] \geq 0.$
This can be applied to every term in the sum to establish~\eqref{ineq_up}. 
A symmetric argument proves~\eqref{ineq_down}.
%In particular, 
%
%Then 
%\begin{align*}
%\prod_i \lvert y - x + c_i \rvert & \geq \prod_i \lvert x - y + c_i \rvert, \\
%f(x) & \leq f(y),
%\end{align*}
%so that $\lvert \phi(x, y) \rvert \geq \lvert \phi(y, x) \rvert$ and since $\phi(x, y)$ is always negative on $y \geq x$ this is sufficient to show $\phi(y, x) \geq -\phi(x, y)$ for all $y \geq x$. 
%
%This argument can be applied to every term in the sum~\eqref{decomposition} and we obtain
%\[\E_x[h(S(1)) 1_{\{\tau = 1\}})] \geq 0, \quad x \in W^d. 
%\]
%As discussed above this is sufficient to prove the statement.
\end{proof}

\section{Lower bound on the harmonic function}
\label{sec:lower_bound}
%
%Throughout we will assume that there exists 
%$t_0>0$ such that 
%\[
%\E[\mathrm{e} ^{t X_{11} } ]<\infty, \quad 
%|t|<t_0.    
%\]

For any $\veps>0$, recall
\[
	W_{t,\veps}=\left\{x\in \mathbb{R}^d:x_{j+1}-x_{j}>t^{1/2-\veps},1\le j< d\right\}.
\]
It follows from the definition of $h$ that $h(x)\ge \Delta(x)$ for $x \in W^d$. 
We will show now that on $x\in W_{t,\varepsilon}$ these functions are of the same order. 
\begin{lemma}\label{lem.delta.h}
Suppose $a \in (0, 1/4)$ and $\veps > 0$ satisfy that $1/2 -\veps > 2a$. There exists $t_0>0$ and $c> 0$ such that for $t > t_0, \gamma\in(0,1/2-2a-\varepsilon)$ and $x\in W_{t,\varepsilon}$, 
  \[
  \Delta(x)\ge \left(1-\frac{c}{t^\gamma}\right)h(x).   
  \]  
  \end{lemma}

\begin{proof}
  We have, using the inequality $1+x \le e^x$ 
  and $x_j-x_i\ge (j-i)t^{1/2-\varepsilon}$ following from the assumption $x\in W_{t,\varepsilon}$,  
  \begin{align*}
    h(x) &= \E\left[\prod_{i<j} (x_j-x_i+\eta_j-\eta_i) \right]
    \le \Delta(x) \E\left[\prod_{i<j} \left(1+\frac{\eta_j-\eta_i}{(j-i)t^{1/2-\varepsilon}}\right) \right]\\ 
    &\le \Delta(x) \E\left[ \exp\left(\sum_{i<j}\frac{\eta_j-\eta_i}{(j-i)t^{1/2-\varepsilon}}\right) \right].
  \end{align*}   
  By the definition of $\eta_j$, 
  \begin{align*}
    \sum_{i<j}\frac{\eta_j-\eta_i}{j-i}
    &\le 
    \sum_{i<j} \frac{1}{j-i}\sum_{k=i}^{j-1} \zeta_k
    =\sum_{k=1}^{d-1} \zeta_k \sum_{1\le i\le k<j\le d} \frac{1}{j-i}
    \end{align*}
  Fix $\ell\in\{1,2,\ldots,d-1\}$. The number of pairs $(i,j)$ such that 
  $1\le i\le k<j\le d$ and $j-i=\ell$ does not exceed $\ell$. This implies that 
  $$
  \sum_{1\le i\le k<j\le d} \frac{1}{j-i}\le d-1.
  $$
  Consequently,
  \begin{align*}  
  \sum_{i<j}\frac{\eta_j-\eta_i}{j-i}
    \le d \sum_{k=1}^{d-1} \zeta_k. 
  \end{align*}   
  Hence, 
  \[ 
    h(x)\le \Delta(x) 
    \E\left[ \exp\left(\frac{d}{t^{1/2-\varepsilon}}\sum_{k=1}^{d-1} \zeta_k \right) \right]
    =\Delta(x) 
    \left(\E \exp\left(\frac{d}{t^{1/2-\varepsilon}}\zeta \right)\right)^{d-1}. 
  \] 
  Since $\E [e^{u \zeta}] \le 1+Cu$ for some constant $C$ and all sufficiently small $u$   
  we further obtain 
  \begin{align*}
    h(x)\le \Delta(x) \left(1+C\frac{d}{t^{1/2-\varepsilon}}\right)^{d-1}
  \end{align*}
  implying the claim of the lemma. 
  \end{proof}  
We will also need the following simple statement. 
\begin{lemma}\label{lem:hoelder}
    Let $p\ge 1$. 
  There exists $t_0 > 0$ and a finite $C$ such that for all $t > t_0$ and $x\in W^d$ satisfying $x_{i+1}-x_i\ge 1$ for $1\le i\le d-1$,
    \begin{equation}\label{eq:hoelder}
        \left(\E_x\left[
            |\Delta(S(t))|^p
        \right]\right)^{1/p} \le C\Delta(x)t^{d^2}. 
    \end{equation}
\end{lemma}
\begin{proof}
Using $|a+b|\le (1+|b|)|a|, |a|\ge 1$, 
\begin{align*} 
(\E_x[|\Delta(S(t))|^p])^{1/p} &\le 
\Delta(x)
\left(\E_0\left[\prod_{i<j}\left(1+|S_j(t)-S_i(t)|\right)^p\right]\right)^{1/p}.  
%\le 
%C\Delta(x) t_k^{d^2}
\end{align*}	
Next, 
\begin{align*} 
  \E_0\left[\prod_{i<j}\left(1+|S_j(t)-S_i(t)|\right)^p\right]
  &\le 
  \E_0\left[\prod_{i<j}\left(1+|S_j(t)|+|S_i(t)|\right)^p\right]\\ 
  \le \E_0\left[\prod_{i<j}\left((1+|S_j(t)|)(1+|S_i(t)|)\right)^p\right]
  &\le \E_0(1+|S_1(t)|)^{d^2p}\le C t^{d^2p}, 
\end{align*}  
where we make use of the following exponential inequality: there exists an $r > 0$
such that $\pr_0(|S_1(t)|>u)\le e^{-r u}$ for $u>t$. The exact value of the exponent $r$ does not need to be chosen optimally and plays no role.
\end{proof}
\begin{lemma}\label{lem6}
	Let $a \in (0, 1/4)$ and $\veps > 0$ with $d(t)\le t^a$ for some $2a< \frac{1}{2}-\varepsilon$.	
  There exist $\gamma>0, C>0$ and $t_0>0$ such that 
  for $t>t_0$ and $x\in W_{t,\varepsilon}$, 
  \begin{align*} 
    \E_x[h(S(t));\tau>t]
    \ge 
    \E_x[\Delta(S(t));\tau>t]
    & \ge 
    \left(1-\frac{C}{t^\gamma}\right)\Delta(x)\\
    &\ge 
    \left(1-\frac{2C}{t^\gamma}\right)h(x).   
  \end{align*}
\end{lemma}  

\begin{proof}
The first inequality follows from $h(x)\ge \Delta(x)$ for $x\in W^d$. 
The third inequality follows from Lemma~\ref{lem.delta.h}. 
The rest of the proof is devoted to the second inequality.

Fix $\delta>0,$ which we will define later. 
Let $X'_{ij}$ be independent random variables 
distributed as $X_{ij}$ conditioned on $\{|X_{ij}|\le t^\delta\}$, let  $S'_j(t)$ 
be the corresponding random walks started from $x_j$ and $\tau'$ their exit time from the Weyl chamber.  
Let $n_t(j) = \sup\{n \geq 0: \sum_{i=1}^n e_{ij} \leq t\}$ and 
$n_t = \sum_{j=1}^d n_t(j)$ be the total number of jumps up to time $t$.
Let $(\widehat{X}_k)_{k \geq 1}$ be a relabelling of the $X_{ij}$ according to the time at which the corresponding jump occurs, so that $(\widehat{X}_k)_{1 \leq k \leq n_t}$ is the collection of jumps that happen before time $t$.

%For $j = 1, \ldots, d$ set $U_j(n) = \sum_{i=1}^n X_{ij}$ for $n \geq 1$
%and let $N_j(t)$ be independent standard Poisson processes. Then we have the %representation 
%\[
%S_j(t) = U_j(N_j(t)), \quad t \geq 0.
%\]
%Let $n_t = \sum_{j=1}^d N_j(t)$ denote the total number of jumps up to time $t$. 

%and $S'_{j}(t)$ the interpolated continuous 
%time process. 

%
%Truncate the increments of the random walk at the level $n^\delta$. 
%Now let $X'_{ij}$ be independent random variables 
%distributed as $X_{ij}\mid |X_{ij}|\le n^\delta$, let  $S'_j(n)$ 
%corresponding random walks. 
We have
\begin{align*} 
 & \E_x[\Delta(S(t));\tau>t] \\
  &\ge
  \E_x[\Delta(S(t));\tau>t, \max_{1 \leq k \leq 2d(t) t}|\widehat{X}_{k}|\le t^\delta,  n_t \leq 2d(t)t ]\\  
  &=
 \E_x[\Delta(S'(t));\tau'>t,  n_t \leq 2d(t)t ]  \pr\left(\max_{1 \leq k \leq 2d(t) t}|\widehat{X}_{k}|\le t^\delta \right)
 \\
 &\ge 
 \E_x[\Delta(S'(t));\tau'>t]
 \pr\left(\max_{1 \leq k \leq 2d(t)t}|\widehat{X}_{k}|\le t^\delta \right)\\
 &\hspace{1cm}-\E_x[\Delta(S'(t));\tau'>t, n_t > 2d(t)t ]. 
\end{align*}
We will first estimate the last summand using the H\"{o}lder inequality. Then we use Lemma~\ref{lem:hoelder} with $p=2$
along with the fact that $n_t \sim \text{Po}(d(t)t)$ and applying a Chernoff bound to obtain
\begin{align*}
\E_x[\Delta(S'(t));\tau'>t, n_t > 2d(t)t ]&\le 
\E_x[\Delta(S'(t))^{2}]^{1/2}
\pr( n_t > 2d(t)t )^{1/2}\\
&\le C \Delta(x)t^{d^2}e^{-d(t)t/2}.
\end{align*}

Let $x\in W_{t,\varepsilon}$ and $x^+ = (x_1 , x_2 + 1, \ldots, x_d + d-1)$. 
%Let $\tau$ be the exit time from the Weyl chamber. 
Define \[u(x)=\Delta(x^+)-\Delta(x).\] 
Note that 
\begin{equation}\label{eq:sumxi}
\sum_{i=1}^{d-1}\frac{1}{x_{i+1}-x_i} \le 
\frac{u(x)}{\Delta(x)}.     
\end{equation}
Let 
\[
T_i:=\min\{s>0\colon \text{dist}(S'_{i+1}(s),S'_i(s))\le t^\delta\}    
\]
and $T=\min(T_1,\ldots,T_{d-1})$.
Applying the optional stopping theorem to the stopping time $T \wedge t$, we have
\begin{align*}
\Delta(x) & = \E[\Delta(S'(T \wedge t))]\\
& = \E[\Delta(S'(t)); T > t] + \E[\Delta(S'(T)); T \leq t].
\end{align*}
Combining this with the inequality $\tau' \geq T$, we obtain
\begin{align*}
    \E[\Delta(S'(t)); \tau' > t] \geq \E[\Delta(S'(t)); T > t] 
    = \Delta(x) - \E[\Delta(S'(T)); T \leq t].
\end{align*}
%Observe that since $\Delta(S'(t))$ is a martingale from Corollary 2.2 of %\cite{konig2002}, by the tower property and optional stopping theorem, 
%\begin{align*} 
%  \E_x[\Delta(S'(t));\tau'>t] &\ge 
%  \E_x[\Delta(S'(t));T>t]  \\
%  & = \E_x[\Delta(S'(t))] -  \E_x[\Delta(S'(t));T\le t]\\
%  &=\Delta(x)-  \E_x[\E[\Delta(S'(t)) I(T\le t) \vert \mathcal{F}_T]]\\
%  & =\Delta(x)-  \E_x[I(T \le t)\E[\Delta(S'(t)) \vert \mathcal{F}_T]]\\
%  & = \Delta(x)-  \E_x[\Delta(S'(T)); T \leq t] 
%\end{align*}
%where $(\mathcal{F}_s)_{s \geq 0}$ is given by $\mathcal{F}_s = \sigma(S(r) : r %\leq s)$.
Then, using~\eqref{eq:sumxi}, 
\begin{align*}
    \E_x[\Delta(S'(T));T\le t]
    &=\sum_{i=1}^{d-1}\E_x\left[
        \Delta(S'(T));T=T_i\le t
    \right]
    \\
    &\le 
    t^\delta 
    \E_x\left[
        \Delta(S'(T))\sum_{i=1}^{d-1}
        \frac{I(T=T_i)}{S'_{i+1}(T)-S'_{i}(T)};T\le t
    \right] \\
   &  \le 
    t^\delta 
    \E_x[u(S'(T));T\le t]. 
\end{align*}
Note that $\Delta(x^+)$ is harmonic for the free random walk using the same argument applied to $h$ in Proposition \ref{superharmonic} (taking $\zeta_j = 1$
for $j = 1, \ldots, d$). Therefore, $u(x)$ is harmonic for the free random walk as a difference of harmonic functions and
$u(S(t\wedge T))$ is a martingale. Hence, 
\[
    \E_x[u(S'(T));T\le t]
    \le 
    \E_x[u(S'(t\wedge T))]=u(x).
\]
Next observe that, due to the assumption $x\in W_{t,\varepsilon}$,  
\[
\frac{u(x)}{\Delta(x)}
\le \left(
    1+\frac{1}{t^{1/2-\varepsilon}}
\right)^{\frac{d(d-1)}{2}}-1.    
\]
For $2a<\frac{1}{2}-\varepsilon$
%This gives 
%\[
%    \frac{1}{4}-\frac{\varepsilon}{2} = \frac{\varepsilon}{2}
%\]
%and $a=\frac{1}{8}$.
we have
\[
    \E_x[\Delta'(S(T));T\le t]
    \le  
    t^\delta 
    \Delta(x)  
    \left(
    \left(
        1+\frac{1}{t^{1/2-\varepsilon}}
    \right)^{\frac{d(d-1)}{2}}-1\right) 
    \le \frac{\Delta(x)}{t^{\gamma_1}} t^\delta
\]
for some $\gamma_1 > 0$.
As a result, 
\[ 
    \E_x[\Delta(S'(t));\tau'>t]\ge 
    \Delta(x)(1-t^{\delta-\gamma_1}).    
\]
We are left to note that 
%\textcolor{red}{[Can this now be deleted? $ \prod_{j = 1}^d \pr(n_t(j) \leq 2t) \geq 1 - \frac{1}{t}$) and]}
\[
	\pr\left(\max_{1 \leq k \leq 2td(t)}|\widehat{X}_{k}|\le t^\delta \right) 
	=(1-\pr(|\widehat{X}_{k}|>t^\delta))^{2td(t)} \ge 1-\frac{1}{t}
\]
provided $|X_{ij}|$ and their relabelling as $\widehat{X}_k$ have sufficiently many moments, which is a consequence of our assumptions in Section \ref{sec:rw_assumptions}. Finally, taking $\delta = \gamma_1/2$ and $\gamma = \gamma_1/2$ gives the required statement.
%The lower bound then follows.  
\end{proof}

Define
\[
\nu_t :=\inf\{r\ge 0\colon S(r)\in W_{t,\delta}\}.
\]
\begin{lemma}\label{lem:repulsion}
  Fix   $\delta\in(0,1/2)$. 
There exist   $C>0$,  
$t_0>1$
 such that uniformly in $x\in W^d, d=d(t)$ and $t>t_0$, 
\[ 
	\pr_x(\nu_t>t^{1-\delta},\tau>t^{1-\delta})\le \exp\{-Ct^\delta\}. 
\]
\end{lemma}
\begin{proof} 
Fix $A$, which we will define later and set $b_t=At^{1/2-\delta}$.
	If $x\in W_{t,\delta}$, then there is nothing to prove since $\nu_t=0$. 
	Hence, we will assume that \(x\in W^d\cap W_{t,\delta}^c\).
	Next note that $W^d\cap W_{t,\delta}^c=\cup_{j=1}^{d-1}W_{j,t,\delta}$, where
\[
	W_{j,t,\delta}=\left\{x\in W^d\colon \min_{1\le k< j}(x_{k+1}-x_k)>t^{1/2-\delta}, x_{j+1}-x_j\le  t^{1/2-\delta}\right\}.
\] 
Note the following estimate for $x\in W_{j,t,\delta}$,
which can be proved by the same arguments as Lemma 3 from~\cite{denisov_wachtel16}, 
\begin{align*}
	\pr_x(\nu_t>b_t^2,\tau>b_t^2)&\le 
	\pr_x(\tau>b_t^2)\\ 
  &\le 
	\pr\left(\min_{s\le b_t^2}S_{j+1}(s)-S_j(s)>0\right) \\
				 &\le C\frac{H(x_{j+1}-x_j)}{b_t}\le 
			 C\frac{H(t^{1/2-\delta})}{b_t},
\end{align*}
where $H(x)$ is the renewal function of descending ladder height process 
of the compound Poisson process $S_{j+1}(s)-S_j(s)$.  
Note that $H$ is the same for all $j$. Hence, by subadditivity of the renewal function we obtain the following uniform bound for $x\in W^d\cap W_{t,\delta}$,  
\begin{equation}
	\label{eq:conc} 
	\pr_x(\nu_t>b_t^2,\tau>b_t^2)\le C\frac{1+t^{1/2-\delta}}{b_t}\le e^{-1}, 
\end{equation}
for sufficiently large $A$. 
We can now proceed as follows, 
\begin{align*}
&	\pr_x(\tau>t^{1-\delta},\nu_t>t^{1-\delta}) \\
	&=
\e_x\left[\pr_{S_{t^{1-\delta}-b_t^2}}(\tau>b_t^2,\nu_t>b_t^2),\tau>t^{1-\delta}-b_t^2,\nu_t>t^{1-\delta}-b_t^2 \right]\\
			     &\le e^{-1} \pr_x\left(\tau>t^{1-\delta}-b_t^2,\nu_t>t^{1-\delta}-b_t^2\right), 
\end{align*}
since $S_{t^{1-\delta}-b_t^2}\in W^d\cap W_{t,\delta}^c$ on the event $\{\tau>t^{1-\delta}-b_t^2,\nu_t>t^{1-\delta}-b_t^2\}$. Repeating the procedure $\frac{t}{b_t^2}-1$ times we obtain, 
\[
	\pr_x(\tau>t^{1-\delta},\nu_t>t^{1-\delta})\le \exp\left\{-\frac{t}{b_t^2}+1\right\},
\]
which implies the statement of the lemma. 
\end{proof}

\begin{proof}[Proof of Proposition~\ref{lem:asymptotics.v}]
The upper bound  $V(x)\le (1+o(1))\Delta(x)$ is a consequence of Proposition \ref{superharmonic} and Lemma \ref{lem.delta.h}. It remains to prove the lower bound for $V$.

Let $x\in W_{{T},\varepsilon}$. 
Fix some $\delta \in (\varepsilon,1/2-2a)$ and set 
$t= T^{\frac{1/2-\varepsilon}{1/2-\delta}}$. 
Then $x\in W_{t,\delta}$ and, in particular, Lemma~\ref{lem6} 
is applicable with $\varepsilon=\delta$. 

Consider the  sequence $(t_k)_{k\ge 0}$ defined 
recursively as follows: $t_0=t$ and 
	$t_{k+1}^{1-\delta}=t_k$ for $k\ge 0$.
	We now write applying Markov property and Lemma~\ref{lem6} for $k \geq 0,$  
\begin{align*}
& 	\E_x[\Delta(S(t_{k+1}));\tau>t_{k+1}]
\ge 
\E_x[\Delta(S(t_{k+1}));\tau>t_{k+1},\nu_{t_{k+1}}\le t_k]\\ 
&\ge \left(1-\frac{c}{t_{k+1}^\gamma}\right)
\E_x[\Delta(S(\nu_{t_{k+1}}));\tau>\nu_{t_{k+1}},\nu_{t_{k+1}}\le t_k]. 
\end{align*}
Here, $c$ denotes a positive constant whose value can change from line to line.
Note that by Lemma~\ref{lem.delta.h}, 
\[
\Delta(S(\nu_{t_{k+1}}))\ge 
\left(1-\frac{c}{t_{k+1}^\gamma}\right)
h(S(\nu_{t_{k+1}})) 
\]
for the superharmonic function $h$. 
Also since $h$ is superharmonic and dominates $\Delta$ we have
\begin{align*}
\E_x[h(S(\nu_{t_{k+1}}));\tau>\nu_{t_{k+1}},\nu_{t_{k+1}}\le t_k] 
&\ge 
\E_x[h(S({t_{k}}));\tau>t_{k},\nu_{t_{k+1}}\le t_k] \\
&\ge 
\E_x[\Delta(S({t_{k}}));\tau>t_{k},\nu_{t_{k+1}}\le t_k]. 
\end{align*}	
Hence, 
\begin{align*}
	\E_x[\Delta(S(t_{k+1}));\tau>t_{k+1}]
	&\ge 
\left(1-\frac{c}{t_{k+1}^\gamma}\right)^2 
\E_x[\Delta(S({t_{k}}));\tau>t_{k}]
\\&-
\E_x[\Delta(S({t_{k}}));\tau>t_{k},\nu_{t_{k+1}}> t_k]. 
\end{align*}
We can now apply H\"{o}lder's inequality to the second term to obtain 
\[
\E_x[\Delta(S({t_{k}}));\tau>t_{k},\nu_{t_{k+1}}> t_k]
\le 
(\E_x[|\Delta(S({t_{k}}))|^p])^{1/p} 
\pr_x(\tau>t_{k},\nu_{t_{k+1}}> t_k)^{1/q}
\]
for $1/p+1/q=1$.
Then, by Lemma~\ref{lem:hoelder}, taking $p = q = 2$,
\[
(\E_x[|\Delta(S({t_{k}}))|^p])^{1/p} 
\le C \Delta(x) t_k^{d^2}.
\]
 Applying Lemma~\ref{lem:repulsion} where $\veps$ in Lemma~\ref{lem:repulsion} is chosen to be $\delta$ here,
\[
\pr_x(\tau>t_{k},\nu_{t_{k+1}}> t_k)^{1/q}\le 
e^{-Ct_{k+1}^\delta}. 
\]	
Hence, provided $\delta \in (\veps, 1/2)$ is chosen such that $2a<\delta/(1-\delta)$ we obtain, 
\begin{align*}
	\E_x[\Delta(S(t_{k+1}));\tau>t_{k+1}]
	&\ge 
\left(1-\frac{c}{t_{k+1}^\gamma}\right)^2 
\E_x[\Delta(S({t_{k}}));\tau>t_{k}]
\\&-C\Delta(x) e^{-Ct_{k+1}^{\delta} + T^{2a}\log(t_k)}.
\end{align*}
The condition that $2a < \delta/(1-\delta)$ ensures that the first time in this iteration, $k = 0$, decays to zero.
Now recall that $V(x)=\lim_{t\to \infty}\E_x[\Delta(S(t);\tau>t]$.
Hence, repeating the above iterations 
\begin{multline*} 
	V(x)\ge \prod_{k=0}^\infty 
\left(1-\frac{c}{t_{k+1}^\gamma}\right)^2 \E_x(\Delta(S(T)); \tau > T) 
\\- 
C\Delta(x)\sum_{k=0}^\infty e^{-Ct_{k+1}^{\delta}+T^{2a}\log(t_k)}.
\end{multline*}	
%The sum can be chosen to decay by choosing $\delta > 2\alpha$. 
Finally, Lemma~\ref{lem6} shows that for $x \in W_{T, \varepsilon}$,
\[
\E_x(\Delta(S(T)); \tau > T) \geq (1-T^{-\gamma}) \Delta(x)
\]
for $2a < 1/2 - \varepsilon$
and the lower bound follows. 
\end{proof}
%We can now compute the value of $\alpha$. 
%First it should satisfy $\alpha<\varepsilon/2$ and second it should satisfy 
%$\alpha<\frac{1}{6}-\varepsilon/3$. Hence, 
%\[
%	\frac{5\alpha}{3}<\frac{1}{6},
%\]	
%that is $\alpha<1/10$ should work.
%I think that the above proof can be improved 
%to have $\alpha<\varepsilon$ and  
%$\alpha<\frac{1}{6}-\varepsilon/3$. Hence, 
%\[
%	\frac{4\alpha}{3}<\frac{1}{6},
%\]	
%that is $\alpha<1/8$.
%To obtain a better bound Lemma~\ref{lem6} requires improvements.

\section{Convergence to the Airy line ensemble}
\label{sec:convergence}

\subsection{A brief description of the strategy}
We intend to apply coupling of random walks and Brownian motions along with Proposition \ref{lem:asymptotics.v}. In order to apply this we first wait for the spacing between components of the random walk to exceed the distances at which we can accurately approximate $V$ by $\Delta$.
% such that the errors from the coupling do not affect whether the particles remain ordered. 
We then estimate the maximal and minimal heights attained by the top
and bottom random walks up to this hitting time. Finally, we prove that the difference between functionals of non-intersecting Brownian motion with different initial conditions (where the top and bottom particle satisfy some bounds) can be bounded using a monotonicity property of non-intersecting Brownian motion. Together these steps allow us to reduce the calculation for ordered random walks to non-intersecting Brownian motion started from zero. As a result we can use the known convergence to the Airy process based on exact formulas. In Section~\ref{sec:lemmas} we prove a variety of bounds needed to carry out the strategy above. In Section~\ref{sec:proof_thm}
we complete the proof of Theorem~\ref{thm:airy}.
The main difficulty in all of these steps is to have estimates with uniform control in $d, t$. 

\subsection{Preliminary lemmas}
\label{sec:lemmas}

We start with the following crude estimate. 
\begin{lemma}\label{lem.lower.rough}
For any $x\in W^d$ and $\delta > 0$,
\[
V(x)\ge  C_1\exp(-C_2 d^{4+\delta})
\]
for some absolute constants $C_1$ and $C_2.$
\end{lemma}
\begin{proof}
In the proof, we will introduce a time parameter $t > 0$ satisfying $d(t) \leq t^a$
for $a < 1/4.$
% Using the thinning of Poisson processes it is convenient to consider the following equivalent construction of the random walk process: there is a Poisson process of jumps with intensity $d$ where at each jump we uniformly at random pick a random walk that jumps. 
Let $\varepsilon'$ be such that $\varepsilon'<\frac12 -2a$. Choose $\delta' > 0$ sufficiently small. 

For sufficiently small $c > 0$ there exists  $U, W \subset \mathbb{R}$ such that 
$P(X \in U) > c$, $P(X \in W) > c$ and $\sup\{x: x \in U\} < \inf\{x : x \in W\}$. 
For each of the embedded discrete-time random walks $S_j$ we either accept/reject each proposed step respectively. 
For $i \geq 1$, if $X_{ij} \in U$ we reject the proposed step with probability $c(j-1)/((d-1)\pr(X_{ij} \in U))$ and
if $X_{ij} \in W$ we reject with probability $c(d-j)/((d-1)\pr(X_{ij} \in W))$. 
Otherwise the proposed step is accepted. 
Let $A_t$ denote the event that the $X_{ij}$ are accepted for all $1 \leq i \leq dt^{1/2-\veps' + \delta'}$ and $j = 1, \ldots, d$. Note that $\pr(A_t) = p^{d^2t^{1/2-\veps' + \delta'}}$ for $p := 1 -c$.
Let $X_{ij}'$ denote $X_{ij}$
conditioned on the step being accepted and $S_j'$ the corresponding random walks.

We have for each $j = 1, \ldots, d-1$ and an absolute constant $C$,
\begin{align}
\label{mean_shift}
\E[X_{i,j+1}']-\E[X_{i,j}']\geq 
\frac{C}{d}.
\end{align}
To justify this, let $\mathcal{R}$ denote the event that a proposed step is rejected. Then,
\begin{align*}
&\E[X_{i, j+1}; \mathcal{R}^c]-\E[X_{i, j}; \mathcal{R}^c]
 = \E[X_{i, j}; \mathcal{R}]- \E[X_{i, j+1}; \mathcal{R}] \\
& = \frac{c(j-1)\E[X_{ij}; X_{ij} \in U]}{(d-1)\pr(X_{ij} \in U)} + \frac{c(d-j) 
\E[X_{ij}; X_{ij} \in W]}{(d-1) \pr(X_{ij} \in W)}\\
& \qquad - \frac{c(j+1-1)\E[X_{ij}; X_{ij} \in U]}{(d-1)\pr(X_{ij} \in U)}
-\frac{c(d-j-1) 
\E[X_{ij}; X_{ij} \in W]}{(d-1) \pr(X_{ij} \in W)}\\
& \geq \frac{c}{d-1}\left(\E[X_{ij} \vert X_{ij} \in W] - \E[X_{ij} \vert X_{ij} \in U]\right).
\end{align*}
Using $\sup\{x: x \in U\} < \inf\{x : x \in W\}$, we can obtain \eqref{mean_shift}.
% Suppose firstly that there is a region $[a, b]$ where the density exists and has
% $f_X(x) \geq c > 0$ for all $x \in [a, b]$. We uniformly sample 
% under the two-dimensional graph defined by the density and 
% reject in regions $A_j^c = [(a+b)/2-k(d-j)/(d-1), b-k(d-j)/(d-1)]\times [0, c]$
% for a sufficiently small $k$ such that these regions are subsets of 
% $[a, b]\times[0, c]$.
% Then $\eqref{mean_shift}$ is satisfied. Otherwise for a random variable where there are $a < b$ such that
% $\pr(X = a) = p$ and $\pr(X = b) = q$ then we view this as a projection onto the first co-ordinate of a uniformly chosen random variable in the region $(\{a\} \times [0, p]) \cup (\{b\} \times [0, q]).$
% The rejection region is then chosen to be 
% $A_j^c = \{a\} \times [0, k(j-1)/(d-1)] \cup \{b\} \times [0, k(d-j)/(d-1)]$
% for a sufficiently small $k$
% and the projection onto the first co-ordinate satisfies
% \eqref{mean_shift} (abusing notation by not distinguishing between the random variable and its projection).

Consider the event $D_t$ where 
the $d$-th walker jumps $dt^{1/2-\varepsilon'+\delta'}$ times 
while the other walkers stay in the same place. 
Then after the $k$-th walker has performed its jumps the $(k-1)$-th walker jumps $dt^{1/2-\varepsilon'+\delta'}$ times while the other walkers do not move, and so on.
% Generally, the $k$-th component of the random walk jumps 
% $d t^{1/2-\varepsilon}$ times with $X_{ik}\in A_{k}$ 
% while the other random walks do not move. 
Note that 
\[
\pr(A_{t} \cap D_t) = \left(\frac{p}{d}\right)^{d^2t^{1/2-\varepsilon'+\delta'}}
\ge C_1\exp(-C_2 d^2 t^{1/2-\varepsilon'+\delta''})
\]
for some absolute constants $C_1,C_2$ and $\delta'' > \delta'$. 
Put 
\[
\sigma_t:=\sum_{k=1}^d \sum_{i=1}^{dt^{1/2-\varepsilon'+\delta'} } e_{ik}. 
\]
%Note that $\tau>\sigma_t$ on $B$.  
Moreover, since the conditioned random walks $S'$ satisfy \eqref{mean_shift},
\begin{align}\label{eq:Dt}
\pr(S'(\sigma_t) \in W_{t, \varepsilon'}, \tau>\sigma_t \mid D_t) 
 & \geq 1-\sum_{j=1}^{d-1} \pr(S'_{j+1}(\sigma_t)-S'_j(\sigma_t) < t^{1/2-\varepsilon'} \mid D_t) \nonumber\\
& \qquad -
\sum_{j=1}^{d-1}
\pr(\max_{s\le \sigma_t}S'_j(s)>S'_{j+1}(\sigma_t)\mid D_t) \nonumber\\
 & \geq 1/2
\end{align}
for sufficiently large $t$.
Indeed, putting $U'_j(n) = \sum_{i=1}^n (X_{ij}'-\E[X'_{ij}])$ 
and using~\eqref{mean_shift} 
we can see that 
\begin{multline*}
\sum_{j=1}^{d-1} \pr(S'_{j+1}(\sigma_t)-S'_j(\sigma_t) < t^{1/2-\varepsilon'} \mid D_t)\\
=
\sum_{j=1}^{d-1} \pr(U'_j(dt^{1/2-\varepsilon'+\delta'})>U'_{j+1}(dt^{1/2-\varepsilon'+\delta'})+Ct^{1/2-\varepsilon'+\delta'}-t^{1/2-\varepsilon'}).
\end{multline*}
Also,
\begin{multline*}
\sum_{j=1}^{d-1}
\pr(\max_{s\le \sigma_t}S'_j(s)>S'_{j+1}(\sigma_t)\mid D_t)\\
=\sum_{j=1}^{d-1} 
\pr\left(\max_{k\le dt^{1/2-\varepsilon'+\delta'}}
U'_j(k)>U'_{j+1}(dt^{1/2-\varepsilon'+\delta'})+Ct^{1/2-\varepsilon'+\delta'}\right)
\end{multline*}
Applying Doob's exponential inequality we can see now that both sums converge to $0$ proving~\eqref{eq:Dt} for sufficiently large $t$. 

Note also that 
\[
\pr\left(
\sigma_t
> t
\right)
\le Ce^{-Ct}.
\]
Next note that since $V$ is harmonic, 
\begin{align*}
V(x) &= \E_x[V(S(\sigma_t\wedge t));\tau>\sigma_t\wedge t]\\
&\ge 
\E_x[V(S(\sigma_t));\sigma_t\le  t, \tau>\sigma_t, A_{t}, D_t]\\
& \ge 
\E_x[V(S'(\sigma_t)); \sigma_t\le  t, \tau>\sigma_t,  S'(\sigma_t) \in W_{t, \varepsilon'}  \vert D_t]
\pr(A_{t} \cap D_t).
\end{align*}
Since $\varepsilon'<\frac 12-2a$ , by Proposition~\ref{lem:asymptotics.v} 
with $\varepsilon=\varepsilon'$, for sufficiently large $t$, 
\begin{align*}
V(x)&\ge 
\frac{1}{2}\E_x[\Delta(S'(\sigma_t)); \sigma_t\le  t,\tau>\sigma_t, S'(\sigma_t) \in W_{t, \varepsilon'} \vert D_t] \pr(A_{t} \cap D_t)\\
&\ge \frac{1}{2} (t^{1/2-\varepsilon'})^\frac{d(d+1)}{2}
\pr(\sigma_t\le  t,\tau>\sigma_t, S'(\sigma_t) \in W_{t, \varepsilon'} \mid  D_t)\pr(A_t \cap D_t)\\
& \ge 
\frac{1}{2} \pr(S'(\sigma_t) \in W_{t, \varepsilon'},\tau>\sigma_t \mid  D_t)\pr(A_t \cap D_t)-\pr(\sigma_t>t)\\
& \ge  C_1\exp(-C_2 d^2 t^{1/2-\varepsilon'+\delta''}). 
\end{align*}
We now choose $2a < 1/2 - \veps' = 2a+a\delta''' < 1-2a$ for $a < 1/4$ and
sufficiently small $\delta'''.$
Finally $\delta'', \delta'''$ can be chosen to set $\delta'' + a\delta'''
= \delta$ which gives the stated result.
\end{proof}

Throughout this section we will use the following notation: for any sequence $1 \leq l_1 < \ldots < l_k \leq d,$
\begin{align*}
h_{[l_1, \ldots, l_k]}(x) & = \E\bigg[\prod_{1 \leq i<j \leq d: (i, j) \notin \{l_1, \ldots, l_k\}} (x_j - x_i + \eta_{j} - \eta_i) \bigg]\\
x_{[l_1, \ldots, l_k]} & = (x_1, \ldots, x_{l_1 - 1}, x_{l_1 + 1}, \ldots, x_{l_2 -1}, x_{l_2 + 1}, \ldots, x_{l_k-1}, x_{l_k+1}, \ldots, x_d) \\
S_{[l_1, \ldots, l_k]} & = (S_1, \ldots, S_{l_1 - 1}, S_{l_1 + 1}, \ldots, S_{l_2 -1}, S_{l_2 + 1}, \ldots, S_{l_k - 1}, S_{l_k+1}, \ldots, S_d)
%W_{[l_1, \ldots, l_k]} & = \{x_{[l_1, \ldots, l_k]} : x_{[l_1, \ldots, l_k]}  \in W_{d-k}\}
%W_{[l_1, \ldots, l_k]} = \{x_1 < \ldots < x_{l_1 - 1} < x_{l_1 + 1} < \ldots x_{l_2 -1} < \ldots <x_{l_k}+1 < \ldots < x_d\} 
\end{align*}
and $\tau_{[l_1, \ldots, l_k]}$ is the first exit time for $S_{[l_1, \ldots, l_k]}$ from $W^{d-k}.$
In this section $C$ is a constant that does not depend on $d, T, x$. Its value does not play any role 
and we allow it to change from line to line.

We will make use of the following 
proposition that couples continuous-time random walks with Brownian motion. 
In the discrete-time case this follows from
 Sakhanenko~\cite[Theorem 1]{Sakhanenko06}. 
%Let $F$ be the distribution function of $X_{11}$.  
\begin{proposition}\label{prop.coupling}
Let $\beta > 0.$
There exists $t_0 > 0$ such that for all $t \geq t_0$
there exists a Brownian motion $B_j$ such that 
\[
 \pr(\sup_{s \leq t} \lvert S_j(s) - B_j(s) \rvert \geq t^{\beta}) 
 \leq Ce^{-t^{\beta/2}}, 
\]
% \[
% 	\pr \left(\sup_{t\le n}|S_j(t)-B_j(t)|\ge C_1\alpha x, \max_{1\le i\le n} |X_{ij}|\le y\right) 
%   \le \left(\frac{n\E|X_{11}|^\alpha}{y^\alpha}\right)^{x/y}, 
% \]
where $C$ is an absolute constant.
\end{proposition}
\begin{proof}
For $j = 1, \ldots, d$ set $U_j(n) = \sum_{i=1}^n X_{ij}$
for $n \geq 1$ and let $N_j(t)$ be 
a standard Poisson process. 
Then we have the following representation, 
\[
S_j(t) = U_j(N_j(t)), \quad t \geq 0.
\]
Next note that $(S_j(n))_{n\ge 0}$ is a random walk with increments distributed as $\sum_{i=1}^{N_j(1)} X_{ij}$. 
Hence, by Sakhanenko~\cite[Theorem 1]{Sakhanenko06} there exists a Brownian motion $B_j$ such that for $\delta > 0$
\begin{align*}
& \pr(\sup_{k \leq [t]+1} \lvert S_j(k) - B_j(k) \rvert \geq t^{\beta}/2)\\
& \leq \pr(\sup_{k \leq [t]+1} \lvert S_j(k) - B_j(k) \rvert \geq t^{\beta}/2, \max_{1 \leq i \leq [t]+1} \lvert S_{j}(i)-S_j(i-1) \rvert \leq t^{\delta})\\
 &\hspace{1cm}+ \pr(\max_{1 \leq i \leq [t]+1} \lvert S_{j}(i)-S_j(i-1) \rvert > t^{\delta}))\\
& \leq e^{-ct^{\beta - \delta}} 
+2t\pr(\lvert S_{j}(1) \rvert > t^{\delta})).
%+ e^{-ct^{\delta}}.
\end{align*}

Then, we have 
\begin{align}
\label{cont_skorokhod.another}
&   \pr(\sup_{s \leq t} \lvert S_j(s) - B_j(s) \rvert \geq t^{\beta}) 
 \leq 
\pr(\sup_{k \leq [t]+1} \lvert S_j(k) - B_j(k) \rvert \geq t^{\beta}/2) \\
& \qquad +\pr(\sup_{k\le [t]+1} \sup_{s\in [0,1]} |S_j(k+s)-S_j(k)|>t^\beta/4) \nonumber\\
& \qquad +\pr(\sup_{k\le [t]+1} \sup_{s\in [0,1]} |B_j(k+s)-B_j(k)|>t^\beta/4)
 \nonumber
 \\ & 
 \le 
 e^{-ct^{\beta - \delta}} 
+2t\pr(\lvert S_{j}(1) \rvert > t^{\delta})) \nonumber\\
& \qquad +2t \pr(\sup_{s\in [0,1]} |S_j(s)|>t^\beta/4)
 +2t\pr(\sup_{s\in [0,1]} |B_j(s)|>t^\beta/4)
\nonumber\\
& \leq e^{-ct^{\beta - \delta}} + e^{-ct^{\delta}}
+e^{-ct^\beta}
\nonumber. 
\end{align}
Here, we make use of the exponential martingale 
of the compound Poisson processes for $s\in [0,1]$ and sufficiently small $\theta$
\[
\frac{e^{\theta S_j(s)}}{\E e^{\theta S_j(s)}}
= e^{\theta S_j(s) - s\E[e^{\theta X_{ij}}]+s}
\]
and the Doob  inequality. 
Taking $\delta=\frac{\beta}{2}$ we arrive at the conclusion.
\end{proof}

Let $\mathcal C_t:=\cap_{j=1}^d \mathcal C_{j,t}$, where for $\beta > 0$,
\[
\mathcal C_{j,t}:=\left\{\sup_{0\le s\le t}|B_j(s)-S_j(s)|\le t^{\beta} \right\}.
\]
Recall $\tau^{\text{bm}} = \inf\{t \geq 0: B(t) \notin W^d\}$. 
\begin{lemma}
\label{lem:error_bounds1}
Let $a < 1/2$, $\beta > 0$ and $\theta > \max(1, \beta/2)$.
\begin{enumerate}
    \item Suppose $\veps \in (0, 1/2)$ and $2a < 1/2 - \veps$.
There exists $t_0 > 0$ and $C > 0$ such that uniformly in $d(t) \leq t^a$ and $t > t_0$, 
for all $x \in W_{t,\varepsilon}$ with $\lvert x_j \rvert \leq t^{\theta}/\log(t)$,
\begin{align*}
\E_x[V(S(t)); \tau > t, \mathcal C_t^c] & \leq CV(x) e^{\theta t^{a} \log(2t)-t^{\beta/2}}.
\end{align*}
\item
There exists $t_0 > 0$ and $C > 0$ such that uniformly in $d(t) \leq t^a$ and $t > t_0$,
 for all $x \in W^d$ with $\lvert x_j \rvert \leq t^{\theta}/\log(t)$ and $x_j - x_{j-1} \geq 1$ for all $j = 2, \ldots, d,$
\begin{align*}
\E_{x}[\Delta(B(t)); \tau^{\mathrm{bm}} > t, \mathcal C_t^c] & \leq C\Delta(x) e^{\theta t^a \log(2t)-t^{\beta/2}}.
\end{align*}
\end{enumerate}
\end{lemma}

\begin{proof}
Let
\[
	\mathcal T_t:= \{S_d(t) - S_1(t) \leq t^\theta\}
	%\cap_{j=1}^d\{|S_j(n)|\le n\}.
\]
and recall the definition of $h$ in terms of $\eta_1, \ldots, \eta_d$ from Section \ref{sec:superharmonic}.
Note that 
\begin{align}
	\label{eq:cjn}
\E_{x}[V(S(t));\tau>t,\mathcal C_t^c]
\le\sum_{j=1}^d \E_{x}[V(S(t));\tau>t,\mathcal C_{j,t}^c, \mathcal T_t, \eta_d \leq t^{\theta}] \nonumber \\+ 
\E_{x}[V(S(t));\tau>t,\mathcal T_t^c]
+ \E_{x}[V(S(t));\tau>t, \eta_d > t^{\theta}, \mathcal T_t]. 
\end{align}
%Denote $h_{[j]}(x):=\prod_{1\le i<k\le d, i\neq j, k\neq j} (x_k-x_i + \eta_{k+1} - \eta_i)$.
%Let $W_j:=\{x\colon x_1<\ldots<x_{j-1}<x_{j+1}<\ldots x_d\}$ be 
%$(d-1)$-dimensional Weyl chamber and $\tau_{[j]}$ the exit time of 
%$S' = (S_1, \ldots, S_{j-1}, S_{j+1}, \ldots, S_d)$ from this chamber. 
If $\eta_d \leq t^{\theta}$ and $S_d(t) - S_1(t) \leq t^\theta$ then $h(S(t)) \leq (2t)^{\theta d} h_{[j]}(S(t))$. Therefore,
\begin{align*}
\E_{x}[V(S(t));\tau>t,\mathcal C_{j,t}^c, \mathcal T_t, \eta_{d} \leq t^{\theta}] 
& \leq \E_{x}[h(S(t));\tau>t,\mathcal C_{j,t}^c, \mathcal T_t, \eta_{d} \leq t^{\theta}]\\
& \le 
(2t)^{\theta d} \E_{x}[h_{[j]}(S(t));\tau>t,\mathcal C_{j,t}^c]\\ 
& \le 
(2t)^{\theta d} \E_{x_{[j]}}[h_{[j]}(S_{[j]}(t));\tau_{[j]}>t]\pr(\mathcal C_{j,t}^c)\\
& \leq h_{[j]}(x_{[j]}) (2t)^{\theta d}  
\pr (\mathcal C_{j,t}^c),
\end{align*}
where we used the fact that  $h_{[j]}(S_{[j]}(t))1(\tau_{[j]}>t)$ is a supermartingale. 
Now applying Proposition~\ref{prop.coupling}, \[
	\pr(\mathcal C_{j,t}^c) \le C e^{-t^{\beta/2}}.
\]
%provided $\E|X_{11}|^{1/\delta}<\infty$. 
Then, 
\begin{align*}
\E_{x}[V(S(t));\tau>t,\mathcal C_{j,t}^c, \mathcal T_t, \eta_d \leq t^{\theta}] & \le 
C h_{[j]}(x)\exp\left(\theta t^a\log(2t)-t^{\beta/2}\right) \\
& \leq  C h(x)\exp\left(\theta t^a\log(2t)-t^{\beta/2}\right).
\end{align*} for sufficiently large $t$.
Note that $h_{[j]}(x) \leq h(x)$ holds for all $x$ by the condition that the random variables $\zeta_j$ in the definition of $h$ take values in $[1, \infty)$.
Finally we use that $h(x) \leq CV(x)$ since $x \in W_{t, \varepsilon}$ by Proposition~\ref{lem:asymptotics.v} and Lemma~\ref{lem.delta.h}.

Consider the second term in \eqref{eq:cjn}. Since $x \in W_{t, \varepsilon}$
we have $V(x) \leq h(x) \leq\ C\Delta(x)$ by Proposition~\ref{superharmonic} and Lemma~\ref{lem.delta.h} after which we can use the bound $\Delta(x) \leq (x_d - x_1)^{d^2}$ for $x \in W^d$.
% The function $V(x)$ is increasing in differences $(x_j - x_i)$. For every $x \in W$, we set 
% \[
% x_j^* = x_j + (j-1) t^{1/2-\varepsilon}, \quad j = 1, \ldots, d.
% \]
% Since $x^* \in W_{t, \varepsilon}$, 
% \begin{equation}
% \label{eqn:v_bound}
% V(x) \leq V(x^*) \leq C\Delta(x^*) \leq C(x_d - x_1 + d t^{1/2-\varepsilon})^{d^2}.
% \end{equation}
Therefore,
\begin{align*}
& \E_x[V(S(t)); \tau > t, (S_d(t) - S_1(t)) > t^\theta] \\
& \leq C \E_x[(S_d(t) - S_1(t))^{d^2}; (S_d(t) - S_1(t)) > t^\theta].
% & \leq C\exp(d^3 t^{1/2 - \varepsilon}) \E[(S_d(t) - S_1(t))^{d^2};  (S_d(t) - S_1(t)) > \textcolor{blue}{t^c}].
\end{align*}
For every $y \geq t^\theta$ we have $\pr(S_d(t) - S_1(t) \geq y) \leq e^{-ry}$ 
for some $r > 0$ for $\lvert x_j \rvert \leq t^{\theta}/\log(t)$. 
Therefore,
\[
\E_x[(S_d(t) - S_1(t))^{d^2};  (S_d(t) - S_1(t)) > t^{\theta}] \leq Ct^{\theta d^2} e^{-rt^\theta} \rightarrow 0
\]
as $a < 1/2 < \theta/2$.
Finally consider the third term in \eqref{eq:cjn}. For some $r_1 > 0$,
\begin{align*}
\E_x\bigg[\prod_{i < j} (S_j(t) - S_i(t) + \eta_{j} - \eta_i); \eta_d > t^{\theta}, \mathcal T_t\bigg]
& \leq \E[(t + \eta_d)^{d^2}; \eta_d > t^{\theta}]\\
& \leq C t^{\theta d^2} e^{-r_1 t^{\theta}} \rightarrow 0
\end{align*}
since $a < 1/2 < \theta/2$.
The condition $\theta > \beta/2$ ensures the second and third terms in \eqref{eq:cjn} do not contribute at the leading order, 
which allows a simpler form for the right hand side in the final statement in (i). 

%{
%\color{blue}
% To finish the argument we apply Lemma~\ref{lem.lower.rough} that gives 
% \[
% 1\le C_1V(x)\exp(C_2 d^2 t^{1/2-\varepsilon} \log d)  
% \]
% and hence the final result. 
% }

The second equation follows by a similar argument. 
The only difference is that for the analogue of $h_j(x) = o(h(x))$ in the case of the Vandermonde determinant we impose the condition $x_{j+1}-x_j \geq 1$ for all $j = 2, \ldots, d$ but we do not require that $x \in W_{t, \varepsilon}$ and Proposition \ref{lem:asymptotics.v} is not needed.
%For part (i)
%\begin{align*}
%\E_x[V(S(n)); \tau > n, \mathcal C_n^c, \mathcal{A}_n]
%& \leq \E_x[h(S(n)); \tau > n, \mathcal C_n^c, \mathcal{A}_n] \\
%& = (1+o(1))  \E_x[\Delta(S(n)); \tau > n, \mathcal C_n^c, \mathcal{A}_n] 
%\end{align*}
%On the event $\mathcal{A}_n$, $h \sim \Delta$ and the above argument can be applied. 
\end{proof}
Recall \[
\nu_t = \inf\{r \geq 0: S(r) \in W_{t, \veps}\}.
\]
\begin{lemma}
\label{lem:error_bounds2}
Let $a < 1/4$, $\theta > 1$ and $\veps \in (0, 1/2)$. There exists $t_0 > 0, \delta > 0$ and $c, C > 0$ such that uniformly in $d(t) \leq t^a$ and $t > t_0$ and $\varepsilon' < \varepsilon$, for all $x \in W^d$
with $\lvert x_j \rvert \leq t^{\theta}/\log(t)$,
\begin{align*}
\E_{x}[V(S(t)); \tau > t, \nu_t > t^{1-\veps}] & \leq CV(x)
e^{t^{(4+\delta) a} - ct^{\epsilon'}};
\end{align*}
and for all $x \in W^d$ with $x_{j+1} - x_j \geq 1$ for all $j = 2, \ldots, d$,
\begin{align*}
\E_{x}[\Delta(B(t)); \tau > t, \nu_t > t^{1-\veps}] & \leq C\Delta(x)e^{\theta t^{2a}\log(t) - ct^{\epsilon'}}.
\end{align*}
\end{lemma}

\begin{proof}
 {
By H\"{o}lder's inequality, 
\begin{align}
\label{holder_split}
   & \E_x[V(S(t)); \tau > t, \nu_t > t^{1-\varepsilon}] \nonumber \\
    & \leq \E_x[V(S(t))^p; \tau > t]^{1/p}\pr_x(\tau > t, \nu_t > t^{1-\varepsilon})^{1/q}.
 \end{align}
 The function $V(x)$ is increasing in differences $(x_j - x_i)$ by Proposition 4(b) of~\cite{denisov_wachtel10}. The intuition is that $\pr_x(\tau > n) \sim \mathcal{X}V(x)n^{-d(d-1)/4}$ as $n \rightarrow \infty$ from~\cite{denisov_wachtel10}. As $\tau$ is increasing in differences $(x_j - x_i)$, then 
 it is possible to conclude that $V(x)$ is as well. Although \cite{denisov_wachtel10} considers discrete-time random walks, the argument can also be applied to continuous-time random walks.
 
 Choose $\delta' > 0$ such that $2a < 1/2 - \delta'$. For every $x \in W^d$, we set 
\[
x_j^* = x_j + (j-1) t^{1/2-\delta'}, \quad j = 1, \ldots, d.
\]
Since $x^* \in W_{t, \delta'}$, by Proposition~\ref{lem:asymptotics.v},  
\begin{equation}
\label{eqn:v_bound}
V(x) \leq V(x^*) \leq C\Delta(x^*) \leq C(x_d - x_1 + d t^{1/2-\delta'})^{d^2}.
\end{equation}
Therefore, for $a < 1/2$,
\begin{align*}
& (\E_x[V(S(t))^p; (S_d(t) - S_1(t)) \leq t^{\theta}, \tau > t])^{1/p} \\
& \leq C (\E[(S_d(t) - S_1(t) + dt^{1/2-\delta'})^{pd^2}; 0 \leq (S_d(j) - S_1(j)) \leq t^{\theta}])^{1/p}\\
& \leq Ct^{\theta d^2}.
\end{align*}
Then,
\begin{align*}
& \E_x[V(S(t))^p; (S_d(t) - S_1(t)) > t^{\theta}, \tau > t] \\
& \leq C \E[(S_d(t) - S_1(t) + dt^{1/2-\delta'})^{pd^2}; (S_d(t) - S_1(t)) > t^{\theta}] \\
 & \leq C\exp(pd^3 t^{1/2 - \delta'-\theta}) \E[(S_d(t) - S_1(t))^{pd^2};  (S_d(t) - S_1(t)) > t^{\theta}]
\end{align*}
where the final factor can be bounded as in Lemma \ref{lem:error_bounds1} using the bound $\lvert x_j \rvert \leq t^{\theta}/\log(t)$. 
 We now apply Lemma \ref{lem:repulsion}, taking $\veps' = \delta$, to the second factor in \eqref{holder_split} followed by Lemma \ref{lem.lower.rough}
 to establish that
 \begin{align*}
     \E_x[V(S(t)); \tau > t, \nu_t > t^{1-\varepsilon}]    & \leq Ct^{\theta d^2}e^{-ct^{\varepsilon'}}\\
    & \leq CV(x)e^{t^{(4+\delta)a} - ct^{\varepsilon'}}.
\end{align*}
The second statement follows by a similar argument.
}
\end{proof}

For $i = 1, \ldots, d$, let $M_i(t) = \sup_{0 \leq s \leq t} \lvert S_i(s) - S_i(0) \rvert$ and let $\widehat{M}(t) = \sup_{1 \leq i \leq d} M_i(t)$.
\begin{lemma}
\label{lem:error_bounds3}
Let $a < 1/4$, $\delta > 0$, $\veps \in (0, 1)$, $\kappa > 1/2 - \veps/2$ and $\theta > \max(1, 2\kappa - 1 + \veps)$. There exists $t_0 > 0$ such that uniformly in $d(t) \leq t^a$, $t > t_0$ and $x \in W^d$,
with $\lvert x_d - x_1 \rvert \leq t^{\theta}/\log(t)$,  
\[
 \E_x[V(S(\nu_t)); \tau > \nu_t, \nu_t \leq t^{1-\veps}, \widehat{M}(\nu_t) > t^{\kappa}] \leq CV(x) e^{t^{(4+\delta)a} - t^{2\kappa-1+\veps}}.
\]
\end{lemma}
\begin{proof}
Firstly, 
\begin{align*}
   & \E[V(S(\nu_t)); \tau > \nu_t, \nu_t \leq t^{1-\veps}, \widehat{M}(\nu_t) > t^{\kappa}]\\
   & \leq \sum_{i=1}^d \E[V(S(\nu_t)); \tau > \nu_t, \nu_t \leq t^{1-\veps}, 
    M_i(\nu_t) > t^{\kappa}].
\end{align*}
For each $i = 1, \ldots, d$, by the harmonicity of $V$,
\begin{align*}
& \E_x[V(S(\nu_t)); \tau > \nu_t, \nu_t \leq t^{1-\varepsilon}, M_i(\nu_t) > t^{\kappa}]\\
& \leq \E_x[V(S(\nu_t \wedge t^{1-\veps})); \tau > \nu_t \wedge t^{1-\veps}, 
M_i(\nu_t \wedge t^{1-\veps}) > t^{\kappa}] \\
& = \E_x[V(S(t^{1-\veps})); \tau > t^{1-\veps}, 
M_i(\nu_t \wedge t^{1-\veps}) > t^{\kappa}]\\
& \leq \E_x[V(S(t^{1-\varepsilon})); \tau > t^{1-\varepsilon}, M_i(t^{1-\varepsilon}) > t^{\kappa}].
%& \leq \E_x[h(S(n^{1-\varepsilon})); \tau > n^{1-\varepsilon}, M_d(n^{1-\varepsilon}) > \gamma n^{1/2-a/6}]
\end{align*}
For $\delta' > 0$ satisfying $2a < 1/2 - \delta'$ we use~\eqref{eqn:v_bound} to obtain that 
\begin{align*}
& \E_x[V(S(t^{1-\varepsilon})); \tau > t^{1-\varepsilon}, \max_{j \leq t^{1-\varepsilon}}(S_d(j) - S_1(j)) > t^{\theta}] \\
& \leq C \E_x\bigg[\max_{j \leq t^{1-\varepsilon}}(S_d(j) - S_1(j) + d t^{1/2-\delta'})^{d^2};  \max_{j \leq t^{1-\varepsilon}}(S_d(j) - S_1(j)) > t^{\theta}\bigg] \\
& \leq C\exp(d^3 t^{1/2 - \delta'-\theta}) \E_x\bigg[\max_{j \leq t^{1-\varepsilon}}(S_d(j) - S_1(j))^{d^2};  
\max_{j \leq t^{1-\varepsilon}}(S_d(j) - S_1(j)) > t^{\theta}\bigg].
\end{align*}
The random walk $S_d(j) - S_1(j)$ is symmetric so 
\[
\pr_{x_d - x_1}\bigg(\max_{j \leq t^{1-\varepsilon}} (S_d(j) - S_1(j)) \geq x\bigg) \leq 2 \pr_{x_d - x_1}( (S_d(t^{1-\varepsilon}) - S_1(t^{1-\varepsilon})) \geq x)
\]
and for every $u \geq t$ we have $\pr_{x_d - x_1}(S_d(t^{1-\varepsilon}) - S_1(t^{1-\varepsilon}) \geq u) \leq e^{-ru}$ 
for some $r > 0$ and all $\lvert x_d - x_1 \rvert \leq t^{\theta}/\log(t)$. 
Therefore, 
\[
 \E_x\bigg[\max_{j \leq t^{1-\varepsilon}}(S_d(j) - S_1(j))^{d^2};  
\max_{j \leq t^{1-\varepsilon}}(S_d(j) - S_1(j)) > t^{\theta}\bigg] \leq Ct^{\theta d^2} e^{-rt^{\theta}}.
\]
As a result, 
\begin{align}
\label{eqn_Ec}
 & \E_x[V(S(t^{1-\varepsilon})); \tau > t^{1-\varepsilon}, \max_{j \leq t^{1-\varepsilon}}(S_d(j) - S_1(j)) > t^{\theta}] \nonumber \\
& \leq C \exp(d^3 t^{1/2 - \delta'-\theta}) t^{\theta d^2} e^{-rt^{\theta}} \rightarrow 0
\end{align}
for $a < 1/2$.
We next analyse the expectation over the event 
$\mathcal E_t = \{\max_{j \leq t^{1-\varepsilon}} (S_d(j) - S_1(j)) \leq t^{\theta}\}$.
On this set we shall use the bound $V(x) \leq h(x)$.
%\begin{align*}
%& \E_x[V(S(n^{1-\varepsilon})); \tau > n^{1-\varepsilon}, M_d(n^{1-\varepsilon}) > \gamma n^{1/2-a/6}, \mathcal E_n] \\
%& \leq \E_x[h(S(n^{1-\varepsilon})); \tau > n^{1-\varepsilon}, M_d(n^{1-\varepsilon}) > \gamma n^{1/2-a/6}, \mathcal E_n]
%\end{align*}
We first estimate that for some $r_1 > 0$,
\begin{align}
\label{eqn_eta}
&\E\bigg[\prod_{i < j} (S_j(t^{1-\veps}) - S_i(t^{1-\veps}) + \eta_{j} - \eta_i); \eta_d > t^{\theta}, \mathcal E_t \bigg] \nonumber\\
& \qquad \leq \E[(t^{\theta} + \eta_d)^{d^2}; \eta_d > t^{\theta}] 
 \leq C (2t)^{\theta d^2} e^{-r_1 t^{\theta}} \rightarrow 0 
\end{align}
since $a < 1/2$.
If $\eta_d \leq t^{\theta}$ and $S_d(t^{1-\veps}) - S_1(t^{1-\veps}) \leq t^{\theta}$ then 
\[
h(S(t^{1-\veps})) \leq (2t)^{\theta d} h_{[i]}(S_{[i]}(t^{1-\veps})).
\]
%\[
%\prod_{1 \leq i < j \leq d}(S_j(n^{1-\veps}) - S_i(n^{1-\veps}) + \eta_{j+1} - \eta_i)
%\leq (2n)^d \prod_{1 \leq i < j \leq d-1} S_j(n^{1-\veps}) - S_i(n^{1-\veps}) + \eta_{j+1} - \eta_i).
%\]
As a result, we have 
\begin{align*}
& \E_x[h(S(t^{1-\veps}); \tau > t^{1-\veps}, M_i(t^{1-\veps}) > t^{\kappa}, \mathcal E_t, \eta_d \leq t^{\theta}]\\
%\prod_{i< j} (S_j(n^{1-\veps}) - S_i(n^{1-\veps}) + \eta_{j+1} - \eta_i); \eta_d \leq n]; \tau > n^{1-\veps}, E_n, M_d %> n^{1/2 - a/6}] \\
& \leq (2t)^{\theta d} \pr_{x_i}(M_i(t^{1-\veps}) > t^{\kappa}) \E_{x_{[i]}}[h_{[i]}(S_{[i]}(t^{1-\veps})); \tau_{[i]} > t^{1-\veps}] \\
& \leq h_{[i]}(x_{[i]}) (2t)^{\theta d} \pr_{x_i}(M_i(t^{1-\veps}) > t^{\kappa}). 
\end{align*}
By an exponential bound for $M_i$, for some $r_2 > 0$,
%and all $\lvert x_d \rvert \leq t^{\theta},$
\[
\pr_{x_i}(M_i(t^{1-\veps}) > t^{\kappa}) \leq \exp(-r_2 t^{2\kappa - 1+\veps}).
\]
The same exponent holds for all $i = 1, \ldots, d.$
As a result we have 
\begin{multline*}
\E_x[V(S(t^{1-\veps})); \tau > t^{1-\veps}, M_i(t^{1-\veps}) > t^{\kappa}, \mathcal E_t, \eta_d \leq t]\\
\leq C  h_{[i]}(x_{[i]}) e^{\theta t^a \log(2t) - t^{2\kappa-1+\veps}}.
\end{multline*}
%This requires $4/3 a < \veps$.
Then, in combination with \eqref{eqn_eta},
\begin{multline*}
 \E_x[V(S(\nu_t)); \tau > \nu_t, \nu_t \leq t^{1-\veps}, M_i(\nu_t) > t^{\kappa}, \mathcal E_t]\\
 \leq Ch(x) e^{\theta t^a \log(2t) - t^{2\kappa-1+\veps}}
 + C(2t)^{\theta d^2} e^{-r_1 t^{\theta}}.
\end{multline*}
We use that $h(x) \leq C (2t)^{\theta d^2}$ when $\lvert x_d -x_1\rvert \leq t^{\theta}/\log(t)$ which follows from 
\begin{align*}
h(x) & = \E\big[\prod_{i< j}(x_j - x_i+\eta_j - \eta_i); \eta_d \leq t^{\theta}\big]
+ \E\big[\prod_{i< j}(x_j - x_i+\eta_j - \eta_i); \eta_d > t^{\theta}\big]\\
& \leq (2t)^{\theta d^2} + \E[(t^{\theta} + \eta_d)^{d^2}; \eta_d > t^{\theta}]
\end{align*}
where the second term can be bounded as in \eqref{eqn_eta}. 
We combine this with the bound \eqref{eqn_Ec} on the event $\mathcal{E}_t^c.$
Finally, we use that
$V(x) \geq C_1 \exp(-cd^{4+\delta})$ by Lemma \ref{lem.lower.rough}.
\end{proof}

\begin{lemma}
\label{lem:brownian_approx}
Let $\theta > 0$. Uniformly over all $y, z \in W^d$ satisfying $\lvert z_j - y_j \rvert \leq\theta$
for $j = 1, \ldots, d$ we have for uniformly continuous $f:C(k, [-L, L]) \rightarrow [0, 1]$, 
\[
\E_{y}^{(\Delta)}[f(B(s): 0 \leq s)] - \E_z^{(\Delta)}[f(B(s): 0 \leq s)] \leq C \theta.
\]
\end{lemma}

%We now use a monotonicity property of Dyson Brownian motion in order to approximate
%$\E_{S(\nu_n)}^{(\Delta)}[f(B^n)]$ uniformly over initial conditions satisfying 
%$-\theta_n n^{1/2-\epsilon/2} \leq S_1(\nu_n) \leq \ldots \leq S_d(\nu_n) \leq \theta_n n^{1/2-\epsilon/2}$. 

We prove this statement by using the following monotonicity property of non-intersecting Brownian motions. This comes from the representation of non-intersecting Brownian motions as stochastic differential equations, see \cite{grabiner}.
We use Dyson Brownian motion to refer to non-intersecting Brownian motions when we use the SDE representation. Note that Dyson Brownian motion has a strong solution, for example see~\cite{anderson_guionnet_zeitouni_09}.
This representation allows us to couple Dyson Brownian motion with different initial conditions run with the same Brownian motion. In particular,
let $(Y_1(0), \ldots, Y_d(0)) = (y_1, \ldots, y_d) \in W^d$ and $(Z_1(0), \ldots, Z_d(0)) = (z_1, \ldots, z_d) \in W^d$. 
For $j = 1, \ldots, d$ let 
\begin{align*}
dY_j(t) = \sum_{1 \leq i \leq d, i \neq j} \frac{dt}{Y_j(t) - Y_i(t)}
+ dB_j(t) \\
dZ_j(t) = \sum_{1 \leq i \leq d, i \neq j} \frac{dt}{Z_j(t) - Z_i(t)}
+ dB_j(t)
\end{align*}
where we use the same Brownian motions $B_1, \ldots, B_d$ in both systems. 
The strong solution can be extended to the case when the initial conditions are in the closure of $W^d$, see Proposition 4.3.5 of~\cite{anderson_guionnet_zeitouni_09}.

We use Lemma 2.8 of~\cite{aggarwal_huang} which we state 
here. Part (i) also appears as Lemma 4.3.6 of \cite{anderson_guionnet_zeitouni_09} (with strict rather than weak inequalities for the initial condition).
\begin{lemma}[Aggarwal, Huang~\cite{aggarwal_huang}]
\label{dyson_bm_monotonic}
Suppose $(Y_1, \ldots, Y_d)$ and $(Z_1, \ldots, Z_d)$
are coupled Dyson Brownian motions run with the same Brownian motions as described above started from $y_1 \leq \ldots \leq y_d$ and $z_1 \leq \ldots \leq z_d$ respectively.
\begin{enumerate}
    \item Suppose that $y_j \geq z_j$ for all $ j = 1, \ldots, d$. 
Then, for all $t \geq 0$ we have that $Y_j(t)\geq Z_j(t)$ for all $j = 1, \ldots, d$.
    \item Suppose that $y_j - y_i \geq z_j - z_i$ for all $1 \leq i < j \leq d$. 
Then, for all $t \geq 0$ we have that $Y_j(t) - Y_i(t) \geq Z_j(t) - Z_i(t)$ for all $1 \leq i < j \leq d.$
\end{enumerate}
\end{lemma}

%Suppose $(X_1, \ldots, X_d)$ and $(Y_1, \ldots, Y_d)$
%satisfy Dyson Brownian motion run with the same Brownian motion started from the initial conditions 
%$0$ and $S(\nu_n)$. 
%%$(-\theta_n n^{1/2 - \epsilon/2}, \ldots, \theta_n n^{1/2 - \epsilon/2}), S(\nu_n), 
%%(\theta_n n^{1/2 - \epsilon/2}, \ldots, \theta_n n^{1/2 - \epsilon/2})$. 
%Recall that Dyson Brownian motion has a strong solution. 
%The particles are ordered as $X_1 \leq X_2 \leq \ldots \leq X_d$ and $Y_1 \leq Y_2 \leq \ldots \leq Y_d$ and satisfy 
%the system of SDEs for $1 \leq i \leq d$, 
%\begin{align*}
%dX_i(t) & = \sum_{k \neq i} \frac{1}{X_i(t) - X_k(t)} dt + dB_i(t)\\
%dY_i(t) & = \sum_{k \neq i} \frac{1}{Y_i(t) - Y_k(t)} dt + dB_i(t)
%\end{align*}
%Note that Dyson Brownian has a strong solution. 

\begin{proof}[Proof of Lemma~\ref{lem:brownian_approx}]
We have that $z_j-\theta  \leq y_j \leq z_j+\theta$ for $j = 1, \ldots, d.$
%Therefore, we can construct three coupled systems of Dyson Brownian motions started with all particles at $-\theta$, particles at $y$
%and with all particles started from $\theta.$
%Lemma~\ref{dyson_bm_monotonic} shows that the spacings remain larger for all time when particles are started at $y$ than for the other two systems. The initial extremal particles are also ordered and hence the particle positions are ordered in the following way. 
Let $Z$ and $Y$ be coupled systems of Dyson Brownian motion started at $z$ and $y$ respectively and run with the same Brownian motions. We can also consider the vertical translations given by $Z$ run with the same Brownian motions started from $(z_1 - \theta, \ldots, z_d - \theta)$ and $(z_1 + \theta, \ldots, z_d + \theta)$ respectively.
We have
%\[
%Z_j(0) - \theta  \leq Y_j(0) \leq Z_j(0) + \theta .
%\]
%and by Lemma~\ref{dyson_bm_monotonic}
for $j = 1, \ldots, d$ and all $s \geq 0$ by Lemma \ref{dyson_bm_monotonic} part (i),
\[
Z_j(s) - \theta  \leq Y_j(s) \leq Z_j(s) + \theta.
\]
%Alternatively, this could have been established by combining part (ii) of Lemma~\ref{dyson_bm_monotonic} with the fact that the initial extremal particles are ordered.
As 
$f$ is uniformly continuous, 
\[
\E_{y}[f(B)] - \E_z[f(B)] \leq C \theta. \qedhere
\]
\end{proof}
Alternatively, in the case $z = 0$, Lemma~\ref{lem:brownian_approx} could have been established by combining part (ii) of Lemma~\ref{dyson_bm_monotonic} with the fact that the initial extremal particles are ordered.

\subsection{Convergence to the Airy line ensemble}
\label{sec:proof_thm}

\begin{proof}[Proof of Theorem~\ref{thm:airy}]
To establish weak convergence, it is sufficient to prove convergence of the expectation of any uniformly continuous measurable functional $f: D(k, [-L, L]) \rightarrow [0, 1].$

For $i = 1, \ldots, d$ let $\gamma = 1/2-a/6$ and
\[
X_{i}^T(t) = T^{-\gamma}\left( S_{d-i+1}(T+2T^{1-a/3} t) - 2 T^{1/2+a/2} 
- 2T^{1/2 + a/6} t\right).
\]
and 
\[
Y^T_i(t) = T^{-\gamma}\left( B_{d-i+1}(T+2T^{1-a/3} t) - 2 T^{1/2+a/2} 
- 2T^{1/2 + a/6} t\right).
\]

% \[
% X^T_i(t) = 2^{-1/2} T^{-\gamma} \left( S_i(T + 2T^{1-a/3}t) - 2T^{a/2+1/2}\right)
% \]
% and 
% \[
% Y^T_i(t) = 2^{-1/2} T^{-\gamma}  \left( B_i(T+ 2T^{1-a/3}t) - 2T^{a/2+1/2}\right).
% \]

%Let $\ell^T_L = T - 2T^{1-a/3} L$ and 
%Let $r^T_L = T + 2T^{1-a/3}L$,
%let $\rho_T := \inf\{t \geq r^T_L : S(t) \in W_{T, \veps}\}$
%and 
%$\rho_n := \bar{\\rho}_n - r_n$.
%Suppose that $x \in W_{n, \epsilon}$,
%\begin{align*}
%\E_x^{(V)}(f(S(k) : 0 \leq k \leq n) & = \E_x \left[ f(S(k) : 0 \leq k \leq n) \frac{V(S(n))}{V(x)} 1_{\{\tau > n\}} \right] \\
%& = \E_x \left[ f(S(k) : 0 \leq k \leq n) \frac{V(S(n+\gamma_n))}{V(x)} 1_{\{\tau > n + \gamma_n\}} \right]
%\end{align*}
%and let $\mathcal{A}_{2T} = \{S(2T) \in W_{T, \veps}\}$.
Suppose $\beta > 0$. We will introduce further conditions on $\beta$ later in the proof.
Let
$\mathcal C_{2T}:=\cap_{j=1}^d \mathcal C_{j,2T}$ where  
\[
\mathcal C_{j,2T}:=\left\{\sup_{0\le t\le 2T}|B_j(t)-S_j(t)|\le T^{\beta}\right\}.
\]
%Here $S'$ are the interpolated continuous processes.

We first consider the case when $x \in W_{T, \veps}$.
Then,
\begin{align*}
\E_{x}^{(V)} [f(X^T)] &=\frac{1}{V(x)}
\E_{x} [f(X^T)V(S(2T));\tau>2T, \C_{2T}]\\
&\hspace{1cm}+
\frac{1}{V(x)} \E_{x} [f(X^T)V(S(2T));\tau>2T, \C_{2T}^c]. 
\end{align*}
By Lemma~\ref{lem:error_bounds1} the second sum is $o(1)$ for $\frac \beta 2 >a$.
Let $\mathcal{A}_{2T} = \{S(2T) \in W_{T, \veps}\}$ and note that
\begin{align*}
    &\E_{x} [f(X^T)V(S(2T));\tau>2T, \C_{2T}] 
    \geq \E_{x} [f(X^T)V(S(2T));\tau>2T, \A_{2T}, \C_{2T}]. 
\end{align*}
%Next, the first summand, by harmonicity of $V$, can be 
%written as 
%\begin{align*}
%    &\E_{x} [f(X^T)V(S(2T));\tau>2T, \C_{2T}] \\
%    &=\E_{x} [f(X^T)V(S(\rho_T\wedge 2T));\tau>\rho_T\wedge 2T, \C_{2T}]\\
%    &=
%    \E_{x} [f(X^T)V(S(\rho_T));\tau>\rho_T, \A_{2T}, \C_{2T}]
%    \\
%    &\hspace{1cm}+\E_{x} [f(X^T)V(S( 2T));\tau>2T, \A^c_{2T}, \C_{2T}]. 
%\end{align*}
%By Lemma~\ref{lem:error_bounds2} the second term is again $o(1)$ for %$4a<\varepsilon$.
%As a result, 
%\begin{align}
%\label{eqn_f_AC}
%\E_{x}^{(V)} [f(X^T)] \geq 
%\E_{x} [f(X^T)V(S(2T));\tau>2T, \A_{2T}, \C_{2T}]
% +o(1).    
%\end{align}
%Then 
%\[
%\E_{x}^{(V)} [f(X^T)] = \E_{x}^{(V)}%[f(X^T); \A_{2T}, \C_{2T}] + E_1 + E_2.  \]
%where 
%\begin{align*}
%E_1 & \leq \E_{x}^{(V)} [f(X^T); \A_{2T}^c], \\
%E_2 & \leq \E_{x}^{(V)} [f(X^T); \C_{2T}^c].
%\end{align*}
%By the first equations in Lemma~\ref{lem:error_bounds1} and~\ref{lem:error_bounds2} we have that 
%$E_1 \rightarrow 0$ and $E_2 \rightarrow 0$ as $T \rightarrow \infty$ for $ \frac{\beta}{2} > a$ and $\veps > 4a$. 

The function $f:D(k, [-L, L]) \rightarrow [0, 1]$ is uniformly continuous so for some $C > 0$
on the event $\C_{2T}$,
\begin{equation}
\label{f_eqn}
\lvert f(X^T) - f(Y^T) \rvert \leq C \sup_{d-k+1 \leq j \leq d} \sup_{-L \leq t \leq L} \lvert X^T_j(t) - Y^T_j(t) \rvert \leq CT^{\beta - \gamma}.
\end{equation}
Therefore, using this bound and martingale properties
\begin{align*}
& \frac{1}{V(x)} \E_{x}\left[(f(X^T) - f(Y^T)) V(S(2T)); \tau > 2T, \A_{2T}, \C_{2T} \right]\\
& \leq C T^{\beta -\gamma} \frac{1}{V(x)} \E_{x}\left[ V(S(2T)); \tau > 2T, \A_{2T}, \C_{2T} \right] \\
& \leq CT^{\beta - \gamma} \rightarrow 0
\end{align*}
as long as $\beta < \gamma$.
Let \[
M_1 = \frac{1}{V(x)} \E_{x}[ f(Y^T) V(S(2T))); \tau >  2T,  \A_{2T}, \C_{2T}].
\]
%From the above
%\begin{align*}
% \E_{x}^{(V)}( f(X^n); \A_n, \C_n) = M_1
% + O(n^{a/6+\beta -1/2}).
%\end{align*}
On the event $\C_{2T}$ the following inequalities hold: for all $0 \leq t \leq 2T,$
\begin{equation}
\label{rw_bm_diff}
\lvert (S_j(t) - S_{j-1}(t)) - (B_j(t) - B_{j-1}(t)) \rvert \leq 2T^{\beta}.
\end{equation}
%\[
%B_j(k) - B_{j-1}(k) - 2n^{a/6-1/2+\beta} \leq S_j(k) - S_{j-1}(k)
%\leq B_j(k) - B_{j-1}(k) + 2n^{a/6 -1/2 + \beta}.
%\]
Let
\begin{equation}
    \label{x_plus}
    x^{-} = (x_1, x_2 - 2T^{\beta}, \ldots, x_d - 2(d-1)T^{\beta}).
\end{equation}
Note, for sufficiently large $T$, then $x^{-} \in W_{T, \veps}$ if $\beta < 1/2-\veps$ since $x \in W_{T, \veps}$.
%Therefore, 
%\begin{equation*}
%M_1  = \frac{1}{V(x)} \E_{x}[ f(Y^T) V(S(\rho_T)); \tau > \rho_T,  \A_{2T}, %\C_{2T}] 
%\end{equation*}
We now note that we can construct a coupled process $Y^T_{-}$ driven by the same Brownian motions, where $Y^T_-$ has the same definition as $Y^T$ but with the Brownian motions started from the initial condition $x^-$ instead of $x$. Under the law $\E_x$, note that $Y^T_{-}$ involves Brownian motions started from $x^-$.
The coupled systems $Y^T$ and $Y^T_-$ satisfy an analogue of \eqref{f_eqn} so that by following the argument immediately after \eqref{f_eqn} we can establish
\[
\frac{1}{V(x)} \E_x[ \lvert f(Y^T)-f(Y^T_{-})\rvert V(S(2T)); \tau > 2T,  \A_{2T}, \C_{2T}] \rightarrow 0.
\]
Therefore,
\begin{align*}
    M_1 
& =
 \frac{(1+o(1))}{V(x)} \E_x[ f(Y^T_{-}) V(S(2T)); \tau > 2T,  \A_{2T}, \C_{2T}].
\end{align*}
Let $\tau^{\text{bm}}_{-}$ denote the first exit time from the Weyl chamber for the Brownian motion started from $x^-$. Due to the coupling, we have that
$\{\tau^{\text{bm}}_{-} > 2T\} \subset\{\tau > 2T\}$.
As a result, 
\begin{equation*}
M_1  \geq \frac{(1+o(1))}{V(x)} \E_x[ f(Y^T_{-}) V(S(2T)); \tau^{\text{bm}}_- > 2T,  \A_{2T}, \C_{2T}]. 
\end{equation*}
By definition, when $\A_{2T}$ holds, we have $S(2T) \in W_{T, \veps}$ and $V \sim \Delta$ on $W_{T, \veps}$ for $2a < 1/2-\veps$ by Proposition~\ref{lem:asymptotics.v}.
Furthermore, $\Delta(S(2T))=\Delta(B_-(2T))(1+o(1))$ 
for $2a + \beta < 1/2 - \veps$ by \eqref{rw_bm_diff}, where the Brownian motions $B_-$ are started from $x^-$. Therefore,
\begin{align*}
M_1 & \geq \frac{(1+o(1))}{V(x)} \E_{x^{-}}[ f(Y^T) \Delta(B(2T))); \tau^{\text{bm}} >  2T,  \A_{2T}, \C_{2T}].
\end{align*} 
Let $\mathcal{A}_{2T}^{\text{bm}} = \{\frac{1}{2}B(2T) \in W_{T, \veps}\}$.
On the event $\mathcal{C}_{2T}$, we have that $\mathcal{A}_{2T}^{\text{bm}} \subset \mathcal{A}_{2T}$ for sufficiently large $T$, due to the factor $1/2$ appearing in the definition of $\mathcal{A}_{2T}^{\text{bm}}$ and that $\beta < 1/2 - \veps$.
%\begin{align*}
%M_1
%& \leq \frac{(1+o(1))\Delta(x^+)}{V(x)} \E_{x^{+}}( f(B^n) \Delta(B(r_n+\gamma_n))); \tau^{\text{bm}} > r_n + \gamma_n,  \A_n, \C_n).
%\end{align*}
%The martingale property of $\Delta$ for Brownian motion killed when the co-ordinates become disordered can be applied to give
Therefore,
\begin{align*}
M_1 & \geq \frac{(1+o(1))}{V(x)} \E_{x^{-}}[ f(Y^T) \Delta(B(2T)); \tau^{\text{bm}} > 2T,  \A_{2T}^{\text{bm}}, \C_{2T}] \\
& = (1+o(1))\frac{\Delta(x^-)}{V(x)}  \E_{x^{-}}\left[f(Y^T) \frac{\Delta(B(2T))}{\Delta(x^-)}; \tau^{\text{bm}} > 2T,  \A_{2T}^{\text{bm}}, \C_{2T}\right].
%& = (1+o(1)) \frac{\Delta(x^+)}{V(x)} \E_{x^{+}}^{(\Delta)}( f(B^n);  \A_n, \C_n).
\end{align*}
Using $V(x) = \Delta(x)(1+o(1))$ for $2a < 1/2 -\veps$ and $\Delta(x) =  \Delta(x^-)(1+o(1))$ 
for $2a+\beta < 1/2 - \veps,$
\begin{align*}
M_1 & \geq (1+o(1)) \E_{x^{-}}\left[ f(Y^T) \frac{\Delta(B(2T))}{\Delta(x^-)} ; \tau^{\text{bm}} > 2T, \A_{2T}^{\text{bm}}, \C_{2T}\right] \\
& = (1+o(1))  \E_{x^{-}}\left[ f(Y^T) \frac{\Delta(B(2T))}{\Delta(x^-)}; \tau^{\text{bm}} > 2T, \A_{2T}^{\text{bm}}\right].
\end{align*}
%where in the final line we use again that $V(x) = \Delta(x^+)(1+o(1))$ for $2a < \epsilon -1/2.$
To obtain the second line, we use part (ii) of Lemma~\ref{lem:error_bounds1} 
for $\frac{\beta}{2} > a$. 
We now use Corollary 1.5 of~\cite{ben_arous}, which proves that in the Gaussian Unitary Ensemble of size $d \times d$, the minimal gap between eigenvalues multiplied by $d^{4/3}$ converges to a limiting distribution as $d$ grows to infinity. Hence, with probability tending to one, Dyson Brownian motion started from zero is in $W_{T, \veps}$ at time $2T$ for any 
$\veps > 5a/6.$
This property can then be extended to Dyson Brownian motion started from 
$x^{-}$ using the monotonicity of gaps in part (ii) of Lemma \ref{dyson_bm_monotonic}.
Hence, for $\veps > 5a/6$,
\[
\E_{x^{-}}\left[ f(Y^T) \frac{\Delta(B(2T))}{\Delta(x^-)}; \tau^{\text{bm}} > 2T, (\A_{2T}^{\text{bm}})^c\right] \rightarrow 0, \quad T \rightarrow \infty.
\]
%\begin{align*}
%E_3 & := \E_{x^{+}}^{(\Delta)}[ f(Y^T);  \A_{2T}^c ]   = o(1) \\
%  E_4& :=\E_{x^{+}}^{(\Delta)}[ f(Y^T);  \C_{2T}^c ]  = o(1). 
% \end{align*} 
Therefore, now using Lemma~\ref{lem:brownian_approx},
 \begin{align*}
  M_1 & \geq (1+o(1))\E_{x^{-}}\left[ f(Y^T) \frac{\Delta(B(2T))}{\Delta(x^-)}; \tau^{\text{bm}} > 2T\right] \\
 % & = \E_{x^{-}}^{(\Delta)}[ f(Y^T)]\\
 & = (1+o(1))  \E_{0}^{(\Delta)}[ f(Y^T)]
 \end{align*}
 for $x \in W_{T, \veps}$ and $\lvert x_1 \rvert + \lvert x_d \rvert = o(T^{\gamma}).$
% for $1/2-\veps < \gamma$.
% Let
%\[
%x^- = (x_1, x_2 + 2n^{a/6+\beta -1/2}, \ldots, x_d + 2(d-1)n^{a/6 + \beta -1/2}).
%\]
Therefore 
\[
\E^{(V)}_x[f(X^T)] \geq (1+o(1))  \E_{0}^{(\Delta)}[ f(Y^T)].
\]
Applying the same argument to the function $1-f$ gives 
\[
\E^{(V)}_x[(1-f)(X^T)] \geq (1+o(1))  \E_{0}^{(\Delta)}[ (1-f)(Y^T)]
\]
from which we deduce the equality 
\[
\E^{(V)}_x[f(X^T)] = (1+o(1))  \E_{0}^{(\Delta)}[ f(Y^T)].
\]

% In a similar manner we can show that 
%  \[
% M_1 \geq (1+o(1))  \E_{x^{-}}^{(\Delta)}[ f(Y^T)]
% = (1+o(1))  \E_{0}^{(\Delta)}[ f(Y^T)].
% \]

% 
% \E_{x^{+}}^{(\Delta)}[ f(B^n)] & = \E_{x}^{(\Delta)}[ f(B^n)] + o(1).
% \end{align*}
%From our initial approximation we need to show that 
%\begin{align*}
%E_1 & = \E_{x}^{(V)} (f(X^n); \A_n^c, \C_n)  = o(1) \\
%E_2 & = \E_{x}^{(V)} (f(X^n); \C_n^c)  = o(1). 
%\end{align*}
%The events including $\C_n$ are similar to calculations in Section XX. 
%%\begin{align*}
%%\E_{x_n^{+}}^{(\Delta)}[ f(B^n);  C_n^c ] & \leq \E_{x_n^{+}}^{(\Delta)}[ f(B^n)] \pr_0 [C_n^c ]\\
%%& \leq C  \pr_0 [C_n^c ] \rightarrow 0.
%%\end{align*}
%\begin{align*}
%\E_{x}^{(V)} (f(X^n); \A_n^c) & \leq C \pr_{x}^{(V)} (\A_n^c) \rightarrow 0.
%\end{align*}
%To show that  
%\[
%\E_{x^{+}}^{(\Delta)}[ f(B^n)] = \E_{x}^{(\Delta)}[ f(B^n)] + o(1).
%\]
We now consider the case when $x \notin W_{T, \veps}$.
Let 
\[
\nu_T := \inf\{t \geq 0: S(t) \in W_{T, \veps}\}.
\]
Then by the strong Markov property,
\begin{align*}
\E_x^{(V)}[f(X^T)] &= \E_x^{(V)}[ \E_{S(\nu_T)}^{(V)}[f(X^T(\cdot - \nu_T))]; \nu_T \leq T^{1-\veps}] \\
& \quad + 
\E_x^{(V)}[f(X^T); \nu_T \geq T^{1-\veps}].
\end{align*}
Here we are using the property that $f$ does not depend on the paths before time $T^{1-\veps}$ for sufficiently
large $T$.
By the first equation in Lemma~\ref{lem:error_bounds2} we have for $\veps > 4a$,
\[
 \E_x^{(V)}[f(X^T); \nu_T \geq T^{1-\epsilon}] \rightarrow 0, \quad T \rightarrow \infty.
\]
Let $M(t) = \max_{0 \leq s \leq t} S_d(s)$ and $I(t) = \min_{0 \leq s \leq t} S_1(s)$.
The first term can be split further for some $\delta > 0$ as
\begin{align*}
 &\E_x^{(V)}[\E_{S(\nu_T)}^{(V)}[f(X^T(\cdot - \nu_T))]; \nu_T \leq T^{1-\veps}]  \\
 &=  \E_x^{(V)}[\E_{S(\nu_T)}^{(V)}[f(X^T(\cdot - \nu_T))]; \nu_T \leq T^{1-\veps}, 
 \lvert M(\nu_T) \rvert + \lvert I(\nu_T) \rvert \leq T^{\gamma-\delta}]\\
 % I(\nu_n) \geq -\gamma n^{1/2 - a/6}]\\
 & \quad + E_5 + E_6,
 \end{align*}
 where 
 \begin{align*}
 E_5 & \leq
 \E_x^{(V)}[\E_{S(\nu_T)}^{(V)}[f(X^T(\cdot - \nu_T))]; \nu_T \leq T^{1-\epsilon}, \lvert M(\nu_T) \rvert >  T^{\gamma-\delta}/2]\\
E_6 & \leq  \E_x^{(V)}[\E_{S(\nu_T)}^{(V)}[f(X^T(\cdot - \nu_T))]; \nu_T \leq T^{1-\veps}, \lvert I(\nu_T) \rvert > T^{\gamma-\delta}/2].
\end{align*}
%We need to show
%\[
% \E_x[V(S(\nu_n)); \tau > \nu_n, \nu_n \leq n^{1-\epsilon}, M(\nu_n) > \theta_n n^{1/2-\epsilon/2}] \rightarrow 0.
%\]
Lemma~\ref{lem:error_bounds3} shows that $E_5 \rightarrow 0$ and 
%a symmetric argument shows that 
$E_6 \rightarrow 0$ for $2\gamma-1+\veps > 4a$ where $\delta > 0$ is chosen sufficiently small.

%Let 
%\[
%f(y, k, X^n) = f\left(\frac{y}{\sqrt{n}} 1_{\{t \leq k/n\}} + X^n(t) 1_{\{t > k/n\}}\right).
%\]
%On the event $\{x + S(\nu_n) \in dy, M(\nu_n) \leq \theta_n \sqrt{n}\},$
%\[
%f(y, k, X^n) - f(X^n) = o(1)
%\] 
%uniformly in $\lvert y \rvert \leq \theta_n \sqrt{n}$ and $k \leq n^{1-\epsilon}$. 
%
%Therefore using the approximation in Section XX,
%\begin{align*}
%&\E_x^{(V)}[f(X^n); \nu_n \leq n^{1-\epsilon}; M(\nu_n) \leq \theta_n \sqrt{n}]\\
%& = \E_x^{(V)}\left[\E^{(V)}_{S(\nu_n)} \left[f\left(\frac{S(\nu_n)}{\sqrt{n}} 1_{\{t \leq k/n\}} + X^n(t-k/n) 1_{\{t > k/n\}}\right)\right]; 
%\nu_n \leq n^{1-\epsilon}, M_{\nu_n} \leq \theta_n \sqrt{n}\right] \\
%& = \E_x^{(V)}\left[\E^{(V)}_{S(\nu_n)} \left[f\left(\frac{S(\nu_n)}{\sqrt{n}} 1_{\{t \leq k/n\}} + B^n(t-k/n) 1_{\{t > k/n\}}\right)\right]; 
%\nu_n \leq n^{1-\epsilon}, M_{\nu_n} \leq \theta_n \sqrt{n}\right] 
%\end{align*}
%Now 
%\[
%\E\left[f\left(\frac{y}{\sqrt{n}} 1_{\{t \leq k/n\}} + B^n(t-k/n) 1_{\{t > k/n\}}\right)\right] 
%= \E_0^{(\Delta)}[f(B^n)](1+o(1))
%\]
%by results on the entrance law of Dyson Brownian motion.
As a result, 
\begin{align*}
& \E_x^{(V)}[f(X^T)]\\
%& \E_x^{(V)}[f(X^n); \nu_n \leq n^{1-\epsilon},  \lvert M(\nu_n) + I(\nu_n) \rvert \leq \gamma n^{1/2 - a/6}] \\
& =  \E_x^{(V)}[\E_{S(\nu_T)}^{(V)}[f(X^T(\cdot - \nu_T))]; \nu_T \leq T^{1-\veps},  \lvert M(\nu_T) \rvert + \lvert I(\nu_T) \rvert \leq T^{\gamma-\delta}] +o(1)\\
& = (1+o(1)) \E_x^{(V)}[\E_{0}^{(\Delta)}[f(Y^T(\cdot - \nu_T))]; \nu_T \leq T^{1-\veps},  \lvert M(\nu_T) \rvert + \lvert I(\nu_T) \rvert \leq  T^{\gamma-\delta}]
%& = (1+o(1)) \E_x^{(V)}[\E_{0}^{(\Delta)}[f(B^n)]; \nu_n \leq n^{1-\epsilon}, M(\nu_n) \leq \theta_n 
%n^{1/2-\epsilon/2}, I(\nu_n) \geq -\theta_n n^{1/2-\epsilon/2}]\\
%& = (1+o(1))\E_{0}^{(\Delta)}[f(Y^T)],
\end{align*}
where in the second line we use the approximation of ordered random walks and non-intersecting Brownian motion for starting positions in $W_{T, \veps}$ satisfying the bound 
$\lvert S_d(\nu_T) \rvert + \lvert S_1(\nu_T) \rvert \leq T^{\gamma - \delta}$.
We now show that 
\begin{align}
\label{shift_eqn}
& \E_x^{(V)}[\E_{0}^{(\Delta)}[f(Y^T(\cdot - \nu_T))]; \nu_T \leq T^{1-\veps},  \lvert M(\nu_T) \rvert + \lvert I(\nu_T) \rvert \leq  T^{\gamma-\delta}]  \\
& = (1+o(1))
\E_x^{(V)}[\E_{0}^{(\Delta)}[f(Y^T)]; \nu_T \leq T^{1-\veps},  \lvert M(\nu_T) \rvert + \lvert I(\nu_T) \rvert \leq  T^{\gamma-\delta}].\nonumber
\end{align}
For $\gamma > 1/2 - \veps/2$ and on the event $\{\nu_T \leq T^{1-\veps}\}$ we can use uniform continuity of $f$ to obtain that there exists $0 < \delta' < \gamma - 1/2 + \veps/2$ such that
\begin{align*}
& \lvert f(Y^T(\cdot - \nu_T))- f(Y^T) \rvert \\
& \leq CT^{-\gamma} \sup_{(j, t, s) \in J}
\lvert B_{j}(T+2T^{1-a/3}t-s) - B_{j}(T+2T^{1-a/3}t) \rvert \\
& \leq CT^{-\gamma + 1/2-\veps/2 + \delta'} \rightarrow 0
\end{align*}
where $J = \{(j, t, s) : d-k+1 \leq j \leq d,  -L \leq t \leq L, 0 \leq s \leq T^{1-\veps}\}$.
This bound is the analogue of \eqref{f_eqn}. Hence, \eqref{shift_eqn} can be established by following the argument immediately after \eqref{f_eqn}.
Finally, we use again Lemma~\ref{lem:error_bounds2}
and Lemma~\ref{lem:error_bounds3} to obtain, 
\begin{align*}
& \E_x^{(V)}[\E_{0}^{(\Delta)}[f(Y^T)]; \nu_T \leq T^{1-\veps},  \lvert M(\nu_T) \rvert + \lvert I(\nu_T) \rvert \leq  T^{\gamma-\delta}] \\
& = (1+o(1))\E_{0}^{(\Delta)}[f(Y^T)].
\end{align*}

The conditions needed are $\frac{\beta}{2} > a$, $\veps > 4a$, $\beta < \gamma$, $2a < 1/2 - \veps$, $2a + \beta < 1/2 - \veps$
%$a+\beta-\gamma<1/2-\veps$ 
and $2\gamma - 1 + \veps > 4a$.
Here, $\gamma = 1/2 - a/6$ leads to these conditions being satisfiable if $a<\frac 3{50}$. 
Indeed in this case there exists $\beta>\frac{3}{25}$ such that 
when $\varepsilon = \frac{13}{50}$  all above inequalities are met. 
\end{proof}

\subsection{Law of large numbers}
%Let $\delta > 0$ and define
%\[
%X_T(f) = \sum_{j=1}^d f(T^{-1/2-a/2} S_j(Tt) : t \in [\delta, 1])
%\]
%and 
%\[
%Y_T(f) = \sum_{j=1}^d f(T^{-1/2-a/2} B_j(Tt) : t \in [\delta, 1]).
%\

The proof of Theorem~\ref{thm:brownian_approx} is the same as the proof of Theorem~\ref{thm:airy} with three minor differences. 
The function $f$ has a different form but satisfies an analogue of~\eqref{f_eqn}.
As $g$ and $f$ are uniformly continuous: for some $C > 0$
on the event $\C_{2T}$,
\begin{align*}
 \lvert g(d^{-1} X_T(f)) - g(d^{-1} Y_T(f)) \rvert  
& \leq C
\lvert d^{-1} X_T(f) - d^{-1} Y_T(f) \rvert \\
%& \leq C \sup_{1 \leq j \leq d} \sup_{0 \leq t \leq 1} \lvert X^T_j(t) - Y^n_j(t) \rvert \\
&\leq C \sup_{1 \leq j \leq d} \sup_{\delta \leq t \leq T} \lvert S_j(t) - B_j(t) \rvert \\
& \leq CT^{\beta - \gamma}. 
\end{align*}
Note that $f$ retains the property of not depending on the paths before time $T^{1-\veps}$ for sufficiently
large $T$ and this is the reason for introducing the parameter $\delta > 0$.
The scaling factor $\gamma = a/2 + 1/2$ leads to different conditions at the end of the proof of Theorem~\ref{thm:airy}. Therefore we require the condition $a < 1/16$.

\subsection{Fluctuations of linear statistics}
Again, the proof of Theorem~\ref{thm:brownian_approx_clt} is the same as the proof of Theorem~\ref{thm:airy} with minor differences.
Recall for $\delta > 0$
\begin{align*}
X_T(f) & = \sum_{i=1}^d f\left(T^{-1/2-a/2} S_i(Tt) : \delta \leq t \leq 1\right)\\
Y_T(f) & =  \sum_{i=1}^d f\left(T^{-1/2-a/2} B_i(Tt) : \delta \leq t \leq 1 \right).
\end{align*}
On the event $\C_{2T}$ we have 
\begin{align*}
& \lvert g\left(X_T(f) - \E^{(\Delta)} X_T(f)\right) - g\left(Y_T(f) - \E^{(V)} Y_T(f)\right) \rvert\\
& \leq C \lvert X_T(f) - Y_T(f) \rvert + C \lvert \E^{(\Delta)}  X_T(f) - \E^{(V)} Y_T(f) \rvert \\
& \leq CT^{a+\beta - \gamma}. 
\end{align*}
These extra factors of $T^a$ do not affect the critical conditions. The scaling factor $\gamma = a/2 + 1/2$ and so as in the case of Theorem~\ref{thm:brownian_approx} 
we  require $a < 1/16$.

%\begin{proof}
%\begin{align*}
%& d(X_j(t) - X_i(t) - Y_j(t) + Y_i(t)) \\
%& = \sum_{k \neq i} \frac{1}{X_j(t) - X_k(t)} dt - \sum_{k \neq i} \frac{1}{X_i(t) - X_k(t)} dt 
%- \sum_{k \neq i} \frac{1}{Y_j(t) - Y_k(t)} dt + \sum_{k \neq i} \frac{1}{Y_i(t) - Y_k(t)} dt \\
%& = \sum_{k \neq i} \frac{X_i(t) - X_j(t)}{(X_i(t) - X_k(t))(X_j(t) - X_k(t))} dt - 
%\sum_{k \neq i} \frac{Y_i(t) - Y_j(t)}{(Y_i(t) - Y_k(t))(Y_j(t) - Y_k(t))} dt.
%\end{align*}
%On the set $\{X_j(t) - X_i(t) = Y_j(t) - Y_i(t)\}$, then 
%\begin{align*}
%& \frac{d(X_j(t) - X_i(t) - Y_j(t) + Y_i(t))}{dt} \\
%& = (X_i(t) - X_j(t)) \left(  \sum_{k \neq i} \frac{1}{(X_i(t) - X_k(t))(X_j(t) - X_k(t))}  - 
% \frac{1}{(Y_i(t) - Y_k(t))(Y_j(t) - Y_k(t))}\right) \\
% & \geq 0
%\end{align*}
%if $X_j(t) - X_i(t) \geq Y_j(t) - Y_i(t)$. 
%This ensures that for all $t \geq 0$ we have that $X_j(t) - X_i(t) \geq Y_j(t) - Y_i(t)$ for all $1 \leq i < j \leq d.$
%\end{proof}

\appendix

\section{Likelihood ratio ordering}
\label{appendix:lro}

\begin{proof}[Proof of Lemma~\ref{log_concave}]
Let $\theta \geq 0$ and set $t = (y-x)/(y-x+\theta)$. 
By log-concavity, for $x \leq y$,
\begin{align*}
f(x+\theta) & \geq f(x)^{t} f(y+\theta)^{1-t} \\
f(y) & \geq f(x)^{1-t} f(y+\theta)^{t}.
\end{align*}
Therefore
\[
f(x+\theta) f(y) \geq f(x) f(y+\theta).
\]
This shows that
\begin{equation}
\label{translation_lr}
X-\theta \stackrel{lr}{\leq} X.
\end{equation}
It is a general fact, following directly from the definition, that if $W \stackrel{lr}{\leq} X$ then 
$(W \vert W \in A) \stackrel{lr}{\leq} (X \vert X \in A).$
Hence for all $\theta \geq 0$, 
\[
(X-\theta \vert X - \theta > 0) \stackrel{lr}{\leq} (X \vert X > 0).
\]
As $(-X \vert X < 0)$ is positive and $(X-\theta\vert X > \theta)$
also has a log-concave density then taking an expectation over \eqref{translation_lr},
\[
(X-\theta \vert X - \theta > 0) \stackrel{lr}{\leq} \zeta.
\]
We can repeat the argument with $X$ replaced by $(-X)$ to obtain the statement. 
\end{proof}

\begin{proof}[Proof of Lemma~\ref{stoch_ineq}]We compute
\begin{align*}
 & \E(\phi(U, V)) 
%& = \int \int \phi(x, y) \pr(U \in dx) \pr(V \in dy) \\
  \geq \int \int_{\{y >x\}} \phi(x, y) \pr(U \in dx)  \pr(V \in dy) \\
& \qquad \qquad \qquad \qquad \qquad +  \int \int_{\{x > y\}} \phi(x, y)\pr(U \in dx)  \pr(V \in dy) \\
& = \int \int_{\{y  > x\}} \phi(x, y) \pr(U \in dx)  \pr(V \in dy) + \phi(y, x) \pr(U \in dy)  \pr(V \in dx) \\
%& =  \int \int_{\{y > x\}} \phi(x, y) \P(U \in dx)  \P(V \in dy)  + \phi(y, x) \P(U \in dy)  \P(V \in dx) \\
& \geq  \int \int_{\{y > x\}} \phi(x, y) \big(\pr(U \in dx)  \pr(V \in dy) - \pr(U \in dy)  \pr(V \in dx)\big)\\
& \geq 0.
\end{align*}
Firstly we use that $\phi(x, x) \geq 0, x \in \mathbb{R}$. In the penultimate line we use property (i) of $\phi$. In the final line we use $U \leq_{lr} V$ and property (ii) of $\phi$.
\end{proof}
\bibliographystyle{abbrv}
 \bibliography{ORWgrowing}

\end{document}